\documentclass[a4paper, 11pt]{amsart}
\usepackage[T1]{fontenc}
\usepackage[utf8]{inputenc}
\usepackage{a4wide}
\usepackage{lmodern}
\usepackage{mathtools}
\usepackage{amsthm}
\usepackage{amsfonts}
\usepackage{amssymb}
\usepackage{tikz}
\usepackage{tikz-cd}
\usetikzlibrary{tikzmark}
\usepackage{enumitem}
\usepackage{xcolor}
\usepackage{array}
\usepackage{booktabs}
\usepackage{colortbl}
\usepackage{wrapfig}
\usepackage{ifthen}
\usepackage{tensor}
\usepackage{cutwin}
\usepackage{tabularx}
\usepackage{float}
\usepackage{xspace} 
\usepackage{longtable}
\usepackage{comment}
\usepackage{arydshln}
\usepackage{titletoc}

\usepackage[colorlinks=true, allcolors=blue]{hyperref}

\theoremstyle{plain}
\newtheorem{Th}{Theorem}[section]
\newtheorem{Lemma}[Th]{Lemma}
\newtheorem{Cor}[Th]{Corollary}
\newtheorem{Prop}[Th]{Proposition}

\theoremstyle{definition}
\newtheorem{Def}[Th]{Definition}

\newtheorem{Conj}[Th]{Conjecture}
\newtheorem{Rem}[Th]{Remark}
\newtheorem{?}[Th]{Problem}

\newtheorem{Cons}[Th]{Construction}

\definecolor{chartgray}{gray}{0.4}
\definecolor{darkgreen}{rgb}{0, 0.7, 0}
\definecolor{darkcyan}{rgb}{0, 0.7, 0.7}

\newcommand{\tmf}{\text{tmf}\xspace}
\newcommand{\mmf}{\text{mmf}\xspace}

\newcommand{\cl}{\text{cl}\xspace}

\newcommand{\Be}{\text{Be}\xspace}
\newcommand{\MGL}{\text{MGL}\xspace}
\newcommand{\MU}{\text{MU}\xspace}
\newcommand{\BP}{\text{BP}\xspace}
\newcommand{\Tot}{\text{Tot}\xspace}
\newcommand{\KQ}{\text{KQ}\xspace}
\newcommand{\kq}{\text{kq}\xspace}
\newcommand{\ko}{\text{ko}\xspace}
\newcommand{\kgl}{\text{kgl}\xspace}

\newcommand{\ku}{\text{ku}\xspace}
\newcommand{\MSL}{\text{MSL}\xspace}

\newcommand{\tf}{\tilde{f}}
\newcommand{\ts}{\tilde{s}}
\newcommand{\tGamma}{\tilde{\Gamma}}

\newcommand{\fil}{\text{fil}\xspace}

\newcommand{\Dec}{\text{Déc}\xspace}

\newcommand{\eff}{\text{eff}\xspace}

\newcommand{\nocontentsline}[3]{}
\newcommand\stoptoc{%
	\let\origcontentsline\addcontentsline
	\let\addcontentsline\nocontentsline
}
\newcommand\resumetoc{%
	\let\addcontentsline\origcontentsline
}

\title{Filtrations in $\mathbb{C}$-motivic stable homotopy theory}
\author{Konstantin Emming}
\address{Department of Mathematics, Wayne State University, Detroit, MI 48202, USA}
\email{konstantin.emming@gmail.com}



\begin{document}

\begin{abstract}
    We study the effective, connective, and very effective filtrations in the $\mathbb{C}$-motivic, $2$-complete, cellular, stable homotopy category. We do so by using the filtered spectrum model for this category due to Gheorghe-Isaksen-Krause-Ricka, and in particular the motivic analogue functor $\Gamma_\star$. Then we can express the covers making up the respective filtrations of a nice motivic analogue $\Gamma_\star(X)$ via filtered spectra, and use these to compute the slices. Applying this in the case of $X$ being the sphere spectrum, $\MU$, $\ku$, or an Eilenberg-MacLane spectrum recovers a number of conjectures due to Voevodsky. Applying it to $\ko$ recovers a computation of Ananyevskiy-Röndigs-{\O}stv{\ae}r. We can also apply it to $\tmf$ and compute the effective slices of the motivic modular forms spectrum $\mmf$. We also study the effective slice spectral sequence for $\Gamma_\star(X)$, which turns out to contain the same information as the classical Adams-Novikov spectral sequence for $X$.
\end{abstract}

\maketitle

\setcounter{tocdepth}{2}
\tableofcontents

\section{Introduction}

\stoptoc

Motivic homotopy theory, as developed by Morel and Voevodsky \cite{Mor99} \cite{MV99}, combines ideas from algebraic geometry and algebraic topology. As such, it has seen many celebrated applications in both of these fields. We list some examples: On the side of algebraic geometry, motivic homotopy theory was the fundamental framework for the proof of the Milnor conjecture \cite{Voe03} and the Bloch-Kato conjecture \cite{Voe11}. On the topological side, we mention the progress on the computation of the classical stable homotopy groups of spheres \cite{Isa19} \cite{IWX}. Ideas inspired by motivic homotopy theory, namely synthetic homotopy theory, were also a main ingredient in the resolution of the last Kervaire invariant problem \cite{LWX25}.

\subsection{Voevodsky's slice conjectures}

In this article, we will study ideas in motivic homotopy theory motivated by algebraic geometry, but from a mostly topological viewpoint. Namely, in \cite{Voe02a} Voevodsky introduced what is now called the effective filtration in motivic homotopy theory. The aim was to obtain a motivic spectral sequence similar to the classical Atiyah-Hirzebruch spectral sequence in topology. The classical Atiyah-Hirzebruch spectral sequence relates topological $K$-theory and singular cohomology, whereas the motivic spectral sequence was intended to relate algebraic $K$-theory and motivic cohomology. In \cite{Voe02b} Voevodsky proposed that the motivic spectral sequence can be derived from the effective filtration. In order to study any spectral sequence arising from a filtration of spectra, one first needs to understand the slices of the filtration: Let $Y$ be a motivic spectrum and denote the effective filtration by
\[\dotsc \to f_{q + 1}(Y) \to f_q(Y) \to f_{q - 1}(Y) \to \dotsc .\]
The $q$-th effective slice is then defined via the cofiber sequence
\[f_{q + 1}(Y) \to f_q(Y) \to s_q(Y).\]
In \cite{Voe02a}, a number of conjectures about the structure of these slices are made. The conjectures state that the slices of certain motivic spectra $Y$ are given by purely topological data. Namely, their slices are wedge sums of motivic Eilenberg-MacLane spectra indexed by the $E_2$-page of the classical Adams-Novikov spectral sequence of the Betti realization $\Be(Y)$. We collect some of the conjectures here.

\begin{Conj}[Voevodsky]
    Let $Y$ be one of the following motivic spectra:
    \begin{itemize}
        \item The motivic Eilenberg-MacLane spectrum $M\mathbb{Z}$,
        \item the motivic algebraic cobordism spectrum $\MGL$,
        \item the motivic sphere spectrum $S$.
    \end{itemize}
    Let $E_2^{s, f}$ denote the classical Adams-Novikov spectral sequence $E_2$-page of $\Be(Y)$ in stem $s$ and filtration $f$. Then the $q$-th effective slice of $Y$ is given by
    \[s_q(Y) \simeq \bigvee_{s + f = 2q} \Sigma^{s, q} ME_2^{s, f},\]
    where $M$ is the functor taking a group to the associated motivic Eilenberg-MacLane spectrum.    
\end{Conj}

We note that the equation $s + f = 2q$ describes the $2q$-th antidiagonal in the $(s, f)$-plane, i.e. in a typical Adams-type chart.

These conjectures, and related results, have been an active point of research in past years:
\begin{itemize}
    \item The proof in the case of $\MGL$ is based on ideas of Hopkins-Morel and written up by Hoyois in \cite{Hoy15},
    \item working over an algebraically closed field of characteristic zero the case of the sphere spectrum is proven by Levine in \cite[Section 8]{Lev14}, and for more general base schemes by Röndigs-Spitzweck-{\O}stv{\ae}r in \cite[Theorem 2.12]{RSO19},
    \item for motivic Landweber exact spectra a proof is given by Spitzweck in \cite{Spi10} and \cite{Spi12},
    \item for $\kq$, the very effective cover of hermitian $K$-theory, a proof is given by Ananyevskiy-Röndigs-{\O}stv{\ae}r in \cite[Theorem 3.2]{ARO20},
    \item for $\MSL$, the special linear algebraic cobordism spectrum, a proof by Nandy-Röndigs-Zolotarev can be found in \cite{NRZ26}.
\end{itemize}

The original motivation of the author was to verify Voevodsky's slice conjecture in the case of $\mmf$, the motivic modular forms spectrum, as defined in \cite{GIKR}. Given that $\mmf$ is only known to exist in a $2$-complete, $\mathbb{C}$-motivic context, we will restrict to this setting. But most of \cite{GIKR} can probably also be done in a $p$-complete setting for any prime $p$, as mentioned in \cite[Section 1.1]{GIKR}.

The spectrum $\mmf = \Gamma_\star(\tmf\hspace{0.2mm})$ is defined using the $\mathbb{C}$-motivic analogue functor $\Gamma_\star$ from the classical stable homotopy category to the category of filtered spectra. This functor is symmetric lax monoidal, so it takes values in the subcategory $\text{Mod}_{\Gamma_\star(S)}$ of modules over $\Gamma_\star(S)$, where $S$ is the classical sphere spectrum. In \cite[Theorem 6.12]{GIKR} it is shown that the category $\text{Mod}_{\Gamma_\star(S)}$ is equivalent to the $2$-complete, cellular, $\mathbb{C}$-motivic stable homotopy category. That is why we refer to $\Gamma_\star$ as the $\mathbb{C}$-motivic analogue functor.

Instead of just studying $\mmf$, we will study spectra of the form $\Gamma_\star(X)$. Under certain assumptions, we can compute the slices of such motivic analogues.

\begin{Th}[Corollary \ref{Corollary on slices of motivic analogues}]\phantomsection\label{Theorem on slices of motivic analogues}
    Let $X$ be a bounded below classical spectrum with even $\normalfont{\MU}$ homology. Let $E_2^{s, f}$ denote the Adams-Novikov spectral sequence $E_2$-page of $X$ in stem $s$ and filtration $f$. Then the $q$-th effective slice of $\Gamma_\star(X)$ is given by
    \[s_q(\Gamma_\star(X)) \simeq \bigvee_{s + f = 2q} \Sigma^{s, q} ME_2^{s,f },\]
    where $M$ is the functor taking a group to the associated $\mathbb{C}$-motivic Eilenberg-MacLane spectrum.
\end{Th}

Examples of spectra $X$ satisfying the hypotheses of Theorem \ref{Theorem on slices of motivic analogues} include the classical spectra $S$, $\ko$, $\tmf$, $\MU$, $\ku$, and Eilenberg-MacLane spectra. Their $\mathbb{C}$-motivic analogues $\Gamma_\star(X)$ are $S$, $\kq$, $\mmf$, $\MGL$, $\kgl$, and motivic Eilenberg-MacLane spectra. For a proof of this see \cite[Corollary 7.7]{CQ21}. So beyond computing the slices of $\mmf$, we also reprove Voevodsky's slice conjectures for the aforementioned spectra in the $2$-complete (and likely also $p$-complete) $\mathbb{C}$-motivic setting.

In order to compute the slices of motivic analogues as in Theorem \ref{Theorem on slices of motivic analogues}, we will model the effective covers making up the effective filtration via filtered spectra, see Section \ref{section: Filtered spectrum models for covers of C-motivic analogues}. To show that our models do in fact yield these effective covers, we will first give characterizations of these covers in terms of homotopy groups in Section \ref{section: Homotopy characterizations of C-motivic covers}. Our filtered models will satisfy the homotopy group characterizations by construction. We will also give models for the connective and very effective covers in this way.

Using our explicit model for the effective filtration, we can fully compute the effective slice spectral sequence.

\begin{Th}[Theorem \ref{Theorem on the effective slice spectral sequence of a nice motivic analogue}]\phantomsection\label{introduction version of Theorem on the effective slice spectral sequence of a nice motivic analogue}
	Let $X$ be a bounded below classical spectrum with even $\normalfont{\MU}$ homology. Then the $E_1$-page of the effective slice spectral sequence of $\Gamma_\star(X)$ is given by freely adjoining $\tau$ to the $E_2$-page of the Adams-Novikov spectral sequence of $X$.
	
	Furthermore, the differentials in the effective slice spectral sequence of $\Gamma_\star(X)$ and the Adams-Novikov spectral sequence of $X$ are related as follows: $d_r(x) = \tau^r y$ in the effective slice spectral sequence of $\Gamma_\star(X)$ if and only if $d_{2r + 1}(x) = y$ in the Adams-Novikov spectral sequence of $X$.
\end{Th}

The differentials in the effective slice spectral sequence preserve the weight by construction, so that each weight can be seen as its own spectral sequence. In the case where $X$ is the sphere spectrum and the weight is zero, Theorem \ref{introduction version of Theorem on the effective slice spectral sequence of a nice motivic analogue} recovers a result of Levine \cite[Theorem 1]{Lev15}. We should mention that the cited result of Levine does not assume completion at a prime, whereas in all of our results $2$-completion is implicit.

\subsection{Motivic filtrations}

Because of Voevodsky's slice conjectures, we are mainly interested in studying the effective filtration. However, there are also other filtrations used in motivic homotopy theory. Namely, the connective filtration coming from Morel's homotopy $t$-structure \cite[Section 6.2]{Mor05}, and the very effective filtration introduced by Spitzweck-{\O}stv{\ae}r \cite[Definition 5.5]{SO12}. Since the definitions of all of these filtrations can be phrased in a similar way, see Section \ref{section: C-motivic filtrations}, we will consider all of them in this article.

Regarding the very effective filtration we mention the computation of the very effective slices of $\KQ$ due to Bachmann \cite{Bac17}. Using our filtered spectrum approach, we can also compute the homotopy groups of these very effective slices, see Remark \ref{Remark on Bachmann's computation of the slices of KQ}.

There is also \cite[Remark 2.5]{ARO20} which states that the effective and very effective covers of motivic Landweber exact spectra coincide. We can prove the same result for objects of the form $\Gamma_\star(X)$ where $X$ is Landweber exact in the classical sense, see Corollary \ref{effective and very effective covers of motivic analogues of Landweber exact spectra are equivalent}.

\subsection{The filtered spectrum model for \texorpdfstring{$\mathbb{C}$}{C}-motivic homotopy theory}

We will study objects of the form $\Gamma_\star(X)$. The functor $\Gamma_\star: \textbf{SH}^{\cl} \to \textbf{SH}^{\fil}$ takes as input a classical spectrum and outputs a filtered spectrum. By \cite[Theorem 6.12]{GIKR} there is an equivalence of categories
\[{\normalfont \textbf{SH}^{\mathbb{C}} \simeq \text{Mod}_{\Gamma_\star(S)}},\]
where the left-hand side is the $\mathbb{C}$-motivic cellular $2$-complete stable homotopy category, and the right-hand side is the category of modules over the filtered spectrum $\Gamma_\star(S)$. The functor $\Gamma_\star$ is lax symmetric monoidal. In particular, the objects $\Gamma_\star(X)$ are modules over $\Gamma_\star(S)$, so they correspond to $\mathbb{C}$-motivic spectra under the above equivalence.

We will discuss this filtered spectrum model for $\mathbb{C}$-motivic stable homotopy theory in more detail in Section \ref{section: The filtered spectrum model for the C-motivic stable homotopy category}. For now, let us give a rough idea of how this equivalence comes about: The filtered spectrum $\Gamma_\star(S)$ is constructed in such a way that its homotopy groups can be computed by a spectral sequence that is identical to the $\mathbb{C}$-motivic Adams-Novikov spectral sequence of the $\mathbb{C}$-motivic sphere spectrum. As such, $\Gamma_\star(S)$ and the $\mathbb{C}$-motivic sphere spectrum share the same homotopy groups. This, together with the fact that spheres generate the cellular $\mathbb{C}$-motivic category, leads to the above equivalence.

As we will see in Section \ref{section: Filtered spectrum models for covers of C-motivic analogues}, it is possible to obtain explicit filtered spectrum models for the effective, connective, and very effective covers of a motivic analogue $\Gamma_\star(X)$. Using those models in the effective case we can compute the slices and thus prove Voevodsky's slice conjecture for the class of motivic analogues $\Gamma_\star(X)$ where $X$ is bounded below and has even $\MU$ homology, see Theorem \ref{Theorem on slices of motivic analogues}. The bounded below assumption in Theorem \ref{Theorem on slices of motivic analogues} is essential, whereas the even $\MU$ homology assumption is only needed to get the closed form for the slices as stated. When dropping the even $\MU$ homology assumption we can still compute the homotopy groups of the slices, but there is no evident simple closed form.

Another model for $\mathbb{C}$-motivic stable homotopy theory can be given via $\MU$-synthetic spectra \cite[Theorem 7.34]{Pst23}. Similarly to the motivic analogue functor $\Gamma_\star$ for the filtered spectrum model, there is also an analogue functor $\nu$ for synthetic spectra. The motivic spectrum $\Gamma_\star(X)$ is the $\tau$-completion of $\nu X$ for classical spectra $X$ with even $\MU$ homology \cite[Example 5.42]{vN25}. The spectrum $\nu X$ is $\tau$-complete if and only if $X$ is $\MU$-nilpotent complete \cite[Theorem 4.71 (1)]{vN25}. In particular, this is true if $X$ is bounded below \cite[Example 2.111]{vN25}. So for bounded below $X$ with even $\MU$ homology we have $\Gamma_\star(X) \simeq \nu X$, meaning our results can also be interpreted in terms of synthetic analogues. It would be interesting to study motivic filtrations from the synthetic perspective, but we will not pursue that endeavor here.

\subsection{Future directions}

There are a number of interesting future problems suggested by our work.

\begin{?}
    This article is only concerned with the $\mathbb{C}$-motivic stable homotopy category. One could use the filtered spectrum model of \cite{BHS22} for the $\mathbb{R}$-motivic stable homotopy category to study the effective, connective, and very effective filtrations over $\mathbb{R}$.
\end{?}

\begin{?}\phantomsection\label{Problem on homotopy group characterization of covers over other fields}
    In Section \ref{section: Homotopy characterizations of C-motivic covers} we give characterizations of effective, connective, and very effective covers in terms of their homotopy groups. The first and third of these descriptions are specific to working over $\mathbb{C}$, as they require the map $\tau$ to exist, and they also use the structure of the homotopy of the sphere spectrum as a $\mathbb{Z}[\tau]$-module. It would be interesting to try to give such homotopy group characterizations over other fields.
\end{?}

\begin{?}
    In a similar vein but different direction to Problem \ref{Problem on homotopy group characterization of covers over other fields}, one could instead consider characterizations of effective, connective, or very effective covers not just in the cellular subcategory, but in the entire $\mathbb{C}$-motivic stable homotopy category. In our cellular setting, homotopy groups are enough to consider because in the cellular subcategory homotopy groups detect equivalences. In the entire $\mathbb{C}$-motivic homotopy category equivalences are detected by homotopy sheaves. Is it possible to characterize the covers of a general motivic spectrum by some property of its homotopy sheaf? One could consider this over $\mathbb{C}$ first, but also over other fields.
\end{?}

\subsection{Conventions}

Throughout this manuscript we will be working in the cellular subcategory \cite{DI05} of the $\mathbb{C}$-motivic stable homotopy category. We will use the work of Heard \cite{Hea19} setting up the effective and very effective filtrations in this subcategory. Note that by \cite[Theorem 3.16]{Hea19} the effective filtration of a cellular object defined on the cellular subcategory agrees with the effective filtration of that object when considered in the entire $\mathbb{C}$-motivic stable homotopy category. Therefore, we will not distinguish between the two. While the same statement has not been proven for the connective and very effective filtrations, we will not introduce separate notations for our cellular versions of these functors either.

Everything will be implicitly completed at the prime $2$ because we want to make use of \cite{GIKR}, which is also written in an implicitly $2$-completed context. The arguments of \cite{GIKR} are likely to have odd primary analogues, as is remarked in \cite[Section 1.1]{GIKR}. If such analogues do exist, then our results are likely to hold for odd primes too. Note that completion at a prime is essential for us, as we need the map $\tau$ to exist in the homotopy groups of the $\mathbb{C}$-motivic sphere spectrum, and that requires completion at a prime.

We will work in the setting of $\infty$-categories throughout.

\subsection{Organization}

The organization of this article is as follows:

In Section \ref{section: Notation} we collect some commonly used notation.

In Section \ref{section: The filtered spectrum model for the C-motivic stable homotopy category} we give the necessary background on the filtered spectrum model for $\mathbb{C}$-motivic homotopy theory.

In Section \ref{section: C-motivic filtrations} we recall the effective, connective and very effective filtrations in $\mathbb{C}$-motivic homotopy theory.

In Section \ref{section: Homotopy characterizations of C-motivic covers} we provide characterizations of the respective covers which make up the filtrations of Section \ref{section: C-motivic filtrations} in terms of their homotopy groups.

In Section \ref{section: Filtered spectrum models for covers of C-motivic analogues} we give filtered spectrum models for the covers defined in Section \ref{section: C-motivic filtrations}.

In Section \ref{section: Slices of C-motivic analogues} we analyze the slices of motivic analogues for the filtrations defined in Section \ref{section: C-motivic filtrations} and prove Theorem \ref{Theorem on slices of motivic analogues}.

In Section \ref{section: The effective slice spectral sequence} we compute the effective slice spectral sequence for motivic analogues and prove Theorem \ref{introduction version of Theorem on the effective slice spectral sequence of a nice motivic analogue}.

\subsection{Acknowledgements}

The author would like to thank Dan Isaksen for many helpful discussions on the topics of this article, as well as for comments on earlier drafts.

\resumetoc

\section{Notation}\label{section: Notation}

\begin{tabularx}{400pt}{l|X}
    $\textbf{SH}^{\cl}$ & Classical stable homotopy category\\
    $\textbf{SH}^{\mathbb{C}}$ & Stable homotopy category of $\mathbb{C}$-motivic cellular spectra\\
    $\textbf{SH}^{\fil}$ & Stable homotopy category of filtered spectra\\
    $\text{pt}$ & Trivial spectrum in any stable homotopy category\\
    $S$ & Sphere spectrum in any stable homotopy category\\
    $\tau_{\geq q}$ & classical $q$-connected cover functor\\
    $\pi_*$ & Classical homotopy groups\\
    $\pi_{*, *}$ & Homotopy groups for $\mathbb{C}$-motivic or filtered spectra\\
    $\Gamma_\star$ & $\mathbb{C}$-motivic analogue functor\\
    $M$ & $\mathbb{C}$-motivic Eilenberg-MacLane functor\\
    $\Tot$ & Totalization functor for cosimplicial objects\\
    BKSS & Bousfield-Kan spectral sequence\\
    ANSS & Adams-Novikov spectral sequence\\
    $f_q$ & $q$-th effective cellular $\mathbb{C}$-motivic cover functor\\
    $s_q$ & $q$-th effective cellular $\mathbb{C}$-motivic slice functor\\
    $(-)_{\geq q}$ & $q$-th connective cellular $\mathbb{C}$-motivic cover functor\\
    $(-)_{= q}$ & $q$-th connective cellular $\mathbb{C}$-motivic slice functor\\
    $\tf_q$ & $q$-th very effective cellular $\mathbb{C}$-motivic cover functor\\
    $\ts_q$ & $q$-th very effective cellular $\mathbb{C}$-motivic slice functor\\
    $\Be$ & Complex Betti realization functor\\
\end{tabularx}\\

We will attempt to use the letter $X$ only for objects in the classical stable homotopy category $\textbf{SH}^{\cl}$, and the letter $Y$ only for objects in the $\mathbb{C}$-motivic stable homotopy category $\textbf{SH}^{\mathbb{C}}$ or equivalently for certain filtered spectra.

When grading each page of the Adams-Novikov spectral sequence we use Adams grading, meaning we use bidegrees $(s, f)$ where the letter $s$ stands for stem and the letter $f$ stands for Adams-Novikov filtration. These are the usual coordinate axes in an Adams-type chart. That means $f$ is the homological degree and $s + f$ is the internal degree. The column corresponding to a fixed value of $s$ contributes to the homotopy group $\pi_s$ at $E_\infty$.

\section{The filtered spectrum model for the \texorpdfstring{$\mathbb{C}$}{C}-motivic stable homotopy category}\label{section: The filtered spectrum model for the C-motivic stable homotopy category}

In order to prove our results we need to have a good understanding of the filtered spectrum model for the $\mathbb{C}$-motivic, stable, cellular, $2$-complete homotopy category due to \cite{GIKR}. We recall the basic definitions and results here. Let $\textbf{SH}^{\fil}$ denote the $\infty$-category of filtered spectra as in \cite[Section 2.1]{GIKR}. An object $Y$ in this category consists of a sequence of classical spectra $Y_w \in \textbf{SH}^{\cl}$, $w \in \mathbb{Z}$, together with structure maps decreasing the filtration degree
\[\dotsc \to Y_{w + 1} \to Y_w \to Y_{w-1} \to \dotsc.\]
There is a functor
\[\Gamma_\star: \textbf{SH}^{\cl} \to \textbf{SH}^{\fil}\]
which we will discuss in more detail in a moment. For now, let us denote the subcategory of $\textbf{SH}^{\fil}$ given by those filtered spectra that are $\Gamma_\star(S)$-modules by $\text{Mod}_{\Gamma_\star(S)}$, where $S \in \textbf{SH}^{\cl}$ is the classical sphere spectrum. This subcategory contains for example $\Gamma_\star(X)$ for $X \in \textbf{SH}^{\cl}$ because $\Gamma_\star$ is lax symmetric monoidal. Now we can state the main theorem relating filtered spectra to $\mathbb{C}$-motivic homotopy theory as proven in \cite[Theorem 6.12]{GIKR}.

\begin{Th}[Gheorghe-Isaksen-Krause-Ricka]\phantomsection\label{Equivalence between C-motivic homotopy and filtered spectra}
    There is an equivalence of categories
    \[{\normalfont \textbf{SH}^{\mathbb{C}} \simeq \text{Mod}_{\Gamma_\star(S)}}.\]
\end{Th}
We chose the letter $w$ for the filtration degree of a filtered spectrum because it corresponds to the motivic weight under the above equivalence, and the motivic weight is also usually denoted by $w$.

The precise form of the equivalence of Theorem \ref{Equivalence between C-motivic homotopy and filtered spectra}, i.e. the functors giving rise to it, will not be important for us. The main construction we need to consider is the functor $\Gamma_\star$, whose definition we will now recall. For an object $X \in \textbf{SH}^{\cl}$ the filtration degree $w$ part of the filtered spectrum $\Gamma_\star(X)$ is defined to be
\[\Gamma_w(X) = \Tot(\tau_{\geq 2w}(X \wedge \MU^{\wedge \bullet + 1})).\]
We explain the notation:
\begin{itemize}
    \item $\MU^{\wedge \bullet + 1}$ is the usual cosimplicial object built by using the ring spectrum structure of $\MU$,
    \item $\Tot$ denotes totalization with respect to $\bullet$,
    \item $\tau_{\geq 2w}$ denotes the functor taking a classical spectrum to its $2w$-connective cover. It is to be applied for each value of $\bullet$ separately.
\end{itemize}
The object $\Gamma_\star(X)$ becomes a filtered spectrum via the maps $\Gamma_{w + 1}(X) \to \Gamma_w(X)$ induced by $\tau_{\geq 2w + 2} \to \tau_{\geq 2w}$.

In order to understand and study the spectrum $\Gamma_w(X)$, the following two observations are integral:
\begin{enumerate}
    \item Every totalization comes with a Bousfield-Kan spectral sequence (BKSS) computing its homotopy groups,
    \item if we ignore $\tau_{\geq 2w}$ then the BKSS associated to $\Tot(X \wedge \MU^{\wedge \bullet + 1})$ is the ANSS of $X$.
\end{enumerate}
Carefully writing out the definition of the BKSS in the first point, one sees that $\tau_{\geq 2w}$ has the effect of truncating the ANSS below the antidiagonal of index $2w$ in a typical Adams-type chart.

We will mainly be using this collection of BKSSs and the structure maps to understand $\Gamma_\star(X)$. We provide a schematic picture of this data:

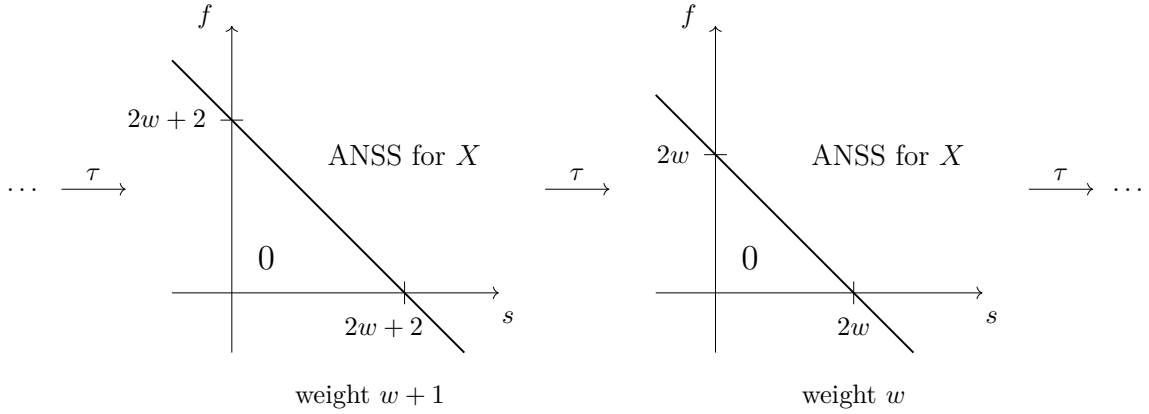
\begin{center}
\resizebox{\textwidth}{!}
{
\begin{tikzpicture}
    \node (dots left) at (4 - 7, 1.5) {$\dotsc$};
    
    \node (arrow left) at (4.4 - 7, 1.5) {};
    \node (arrow right) at (5.6 - 7, 1.5) {};

    \draw[->] (arrow left) -- (arrow right) node[midway,above] {$\tau$};
    
    \node (s0) at (-1, 0) {};
    \node[label=below:{$s$}] (s1) at (4, 0) {};
    \node (f0) at (0, -1) {};
    \node[label=left:{$f$}] (f1) at (0, 4) {};
    
    \draw[->] (s0) to (s1);
    \draw[->] (f0) to (f1);
    
    \node (upper ad) at (-1, 3 + .5) {};
    \node (lower ad) at (3 + .5, -1) {};

    \draw[-, thick] (upper ad) to (lower ad);

    \node[label={[label distance = -.2cm]left:{$2w + 2$}}] (f-axis mark left) at (-.3, 2 + .5) {};
    \node (f-axis mark right) at (.3, 2 + .5) {};
    \node[label={[label distance = -.2cm]below:{$2w + 2\hspace{.6cm}$}}] (s-axis mark lower) at (2 + .5, -.3) {};
    \node (s-axis mark upper) at (2 + .5, .3) {};

    \draw[-] (f-axis mark left) to (f-axis mark right);
    \draw[-] (s-axis mark lower) to (s-axis mark upper);
    
    \node (zero) at (0.5, 0.5) {\Large $0$};
    \node (ANSS) at (2.5, 2) {\large ANSS for $X$};
    \node (weight) at (2, -1.5) {weight $w + 1$};

    \node (middle arrow left) at (4.4, 1.5) {};
    \node (middle arrow right) at (5.6, 1.5) {};

    \draw[->] (middle arrow left) -- (middle arrow right) node[midway,above] {$\tau$};

    \node (s0) at (-1 + 7, 0) {};
    \node[label=below:{$s$}] (s1) at (4 + 7, 0) {};
    \node (f0) at (0 + 7, -1) {};
    \node[label=left:{$f$}] (f1) at (0 + 7, 4) {};
    
    \draw[->] (s0) to (s1);
    \draw[->] (f0) to (f1);
    
    \node (upper ad) at (-1 + 7, 3) {};
    \node (lower ad) at (3 + 7, -1) {};

    \draw[-, thick] (upper ad) to (lower ad);

    \node[label={[label distance = -.2cm]left:{$2w$}}] (f-axis mark left) at (-.3 + 7, 2) {};
    \node (f-axis mark right) at (.3 + 7, 2) {};
    \node[label={[label distance = -.2cm]below:{$2w$}}] (s-axis mark lower) at (2 + 7, -.3) {};
    \node (s-axis mark upper) at (2 + 7, .3) {};

    \draw[-] (f-axis mark left) to (f-axis mark right);
    \draw[-] (s-axis mark lower) to (s-axis mark upper);
    
    \node (zero) at (0.5 + 7, 0.5) {\Large $0$};
    \node (ANSS) at (2.5 + 7, 2) {\large ANSS for $X$};
    \node (weight) at (2 + 7, -1.5) {weight $w$};

    \node (arrow left) at (4.4 + 7, 1.5) {};
    \node (arrow right) at (5.6 + 7, 1.5) {};

    \draw[->] (arrow left) -- (arrow right) node[midway,above] {$\tau$};

    \node (dots right) at (6 + 7, 1.5) {$\dotsc$};
\end{tikzpicture}
}
\end{center}

The horizontal axis denotes the stem $s$ and the vertical axis denotes the Adams-Novikov filtration $f$. We call the filtration degree in $\Gamma_\star(X)$ the weight, since that is what it corresponds to under the equivalence of Theorem \ref{Equivalence between C-motivic homotopy and filtered spectra}. So the left-hand picture describes the BKSS for $\Gamma_{w + 1}(X)$, and the right-hand picture the one for $\Gamma_w(X)$. Below the marked antidiagonals, the spectral sequences are trivial. Above and on the antidiagonals they agree with the ANSS for $X$. The structure maps have been marked with $\tau$. That is because these are maps of spectral sequences that induce $\tau$-multiplication on the homotopy groups of $\Gamma_\star(X)$, where we implicitly use the identification of $\Gamma_\star(X)$ with a $\mathbb{C}$-motivic object via Theorem \ref{Equivalence between C-motivic homotopy and filtered spectra}.

The differentials in the BKSSs are the same as in the ANSS, up to the following case distinction:
\begin{enumerate}
    \item Differentials whose source lies on or above the given antidiagonal occur in the BKSS just like they do in the ANSS,
    \item differentials whose source lies below the given antidiagonal do not occur in the BKSS because their source does not exist.
\end{enumerate}
We illustrate this with a picture:

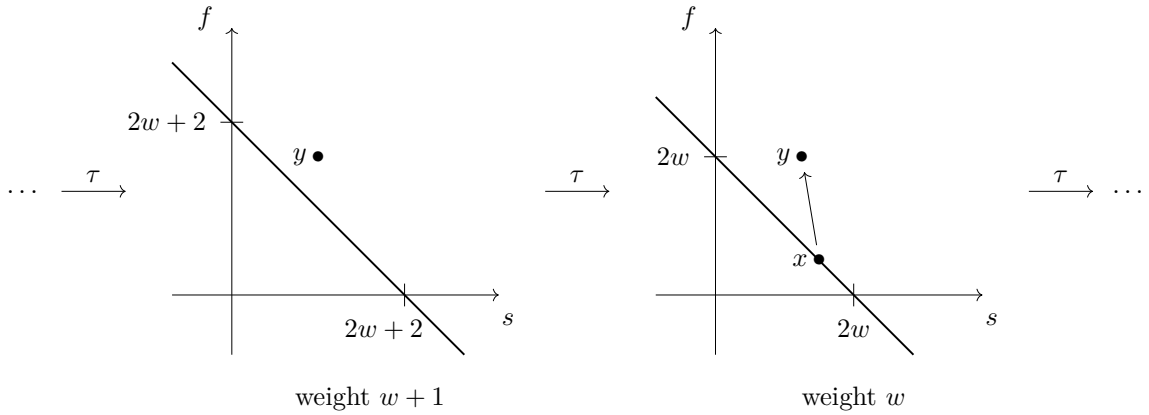
\begin{center}
\resizebox{\textwidth}{!}
{
\begin{tikzpicture}
    \node (dots left) at (4 - 7, 1.5) {$\dotsc$};
    
    \node (left arrow left) at (4.4 - 7, 1.5) {};
    \node (left arrow right) at (5.6 - 7, 1.5) {};

    \draw[->] (left arrow left) -- (left arrow right) node[midway,above] {$\tau$};
    
    \node (s0) at (-1, 0) {};
    \node[label=below:{$s$}] (s1) at (4, 0) {};
    \node (f0) at (0, -1) {};
    \node[label=left:{$f$}] (f1) at (0, 4) {};
    
    \draw[->] (s0) to (s1);
    \draw[->] (f0) to (f1);
    
    \node (upper ad) at (-1, 3 + .5) {};
    \node (lower ad) at (3 + .5, -1) {};

    \draw[-, thick] (upper ad) to (lower ad);

    \node[label={[label distance = -.2cm]left:{$2w + 2$}}] (f-axis mark left) at (-.3, 2 + .5) {};
    \node (f-axis mark right) at (.3, 2 + .5) {};
    \node[label={[label distance = -.2cm]below:{$2w + 2\hspace{.6cm}$}}] (s-axis mark lower) at (2 + .5, -.3) {};
    \node (s-axis mark upper) at (2 + .5, .3) {};

    \draw[-] (f-axis mark left) to (f-axis mark right);
    \draw[-] (s-axis mark lower) to (s-axis mark upper);

    \node (weight) at (2, -1.5) {weight $w + 1$};
    
    \node (middle arrow left) at (4.4, 1.5) {};
    \node (middle arrow right) at (5.6, 1.5) {};

    \draw[->] (middle arrow left) -- (middle arrow right) node[midway,above] {$\tau$};

    \node (s0) at (-1 + 7, 0) {};
    \node[label=below:{$s$}] (s1) at (4 + 7, 0) {};
    \node (f0) at (0 + 7, -1) {};
    \node[label=left:{$f$}] (f1) at (0 + 7, 4) {};
    
    \draw[->] (s0) to (s1);
    \draw[->] (f0) to (f1);
    
    \node (upper ad) at (-1 + 7, 3) {};
    \node (lower ad) at (3 + 7, -1) {};

    \draw[-, thick] (upper ad) to (lower ad);

    \node[label={[label distance = -.2cm]left:{$2w$}}] (f-axis mark left) at (-.3 + 7, 2) {};
    \node (f-axis mark right) at (.3 + 7, 2) {};
    \node[label={[label distance = -.2cm]below:{$2w$}}] (s-axis mark lower) at (2 + 7, -.3) {};
    \node (s-axis mark upper) at (2 + 7, .3) {};

    \draw[-] (f-axis mark left) to (f-axis mark right);
    \draw[-] (s-axis mark lower) to (s-axis mark upper);

    \node (weight) at (2 + 7, -1.5) {weight $w$};

    \node (right arrow left) at (4.4 + 7, 1.5) {};
    \node (right arrow right) at (5.6 + 7, 1.5) {};

    \draw[->] (right arrow left) -- (right arrow right) node[midway,above] {$\tau$};

    \node (dots right) at (6 + 7, 1.5) {$\dotsc$};

    \node[label={[label distance = -.2cm]left:{$y$}}] (left chart target) at (1.25, 2) {$\bullet$};


    \node[label={[label distance = -.2cm]left:{$x$}}] (left chart source) at (1.5 + 7, .5) {$\bullet$};
    \node[label={[label distance = -.2cm]left:{$y$}}] (left chart target) at (1.25 + 7, 2) {$\bullet$};

    \draw[->] (left chart source) to (left chart target);
\end{tikzpicture}
}
\end{center}

In weight $w$ and below, the class $x$ supports a differential, so it disappears. The class $y$ gets hit by a differential, so it is set to $0$ starting on the next page of the BKSS. However, in weight $w + 1$ the element $y$ does not get hit by a differential because $x$ does not exist, so it survives the BKSS and represents a class in homotopy. This class is $\tau$-torsion as it maps to $y$ in weight $w$, which is equal to $0$. Note that $y$ may also exist in weights larger than $w + 1$, so it may be $\tau$-divisible. The largest weight where $y$ exists depends on which antidiagonal $y$ lives on. Equivalently, it depends on the length of the differential from $x$ to $y$.

The main idea of this article is to make adjustments to $\Gamma_\star$ which can be interpreted in terms of pictures similar to the above. From there, it will be relatively simple to deduce that these adjusted versions of $\Gamma_\star$ applied to a classical spectrum $X$ yield effective, connective, or very effective covers of $\Gamma_\star(X)$.

\section{\texorpdfstring{$\mathbb{C}$}{C}-motivic filtrations}\label{section: C-motivic filtrations}

We will recall the definitions of the effective, connective, and very effective filtrations in the context of cellular $\mathbb{C}$-motivic stable homotopy theory, see also \cite{Hea19}. Then, we will collect some basic ways in which one can prove properties of the covers making up these filtrations.

In the rest of this article, a filtration of an object $Y \in \textbf{SH}^{\mathbb{C}}$ will mean a sequence of objects $f_q(Y) \in \textbf{SH}^{\mathbb{C}}$, $q \in \mathbb{Z}$, together with maps decreasing the degree
\[\dotsc \to f_{q + 1}(Y) \to f_q(Y) \to f_{q - 1}(Y) \to \dotsc.\]
Each filtration has its associated slices $s_q(Y) \in \textbf{SH}^{\mathbb{C}}$, defined by the cofiber sequences
\[f_{q + 1}(Y) \to f_q(Y) \to s_q(Y).\]

\begin{Cons}\phantomsection\label{Construction of filtrations using subcategories}
To define our filtrations of interest, we will follow the setup of \cite[Section 2]{Hea19} and refer the reader there for more details. We will give a short account of the main ideas. To define each filtration, we start with a collection of full subcategories of $\textbf{SH}^{\mathbb{C}}$
\[\dotsc \subset \mathcal{C}_{q + 1} \subset \mathcal{C}_{q} \subset \mathcal{C}_{q - 1} \subset \dotsc.\]
For every $q \in \mathbb{Z}$ the inclusion $\mathcal{C}_q \xhookrightarrow{} \textbf{SH}^{\mathbb{C}}$ will have a right adjoint $\textbf{SH}^{\mathbb{C}} \to \mathcal{C}_q$. Consider the composite
\[\textbf{SH}^{\mathbb{C}} \to \mathcal{C}_q \xhookrightarrow{} \textbf{SH}^{\mathbb{C}}\]
and define the $q$-th cover $f_q(Y) \in \textbf{SH}^{\mathbb{C}}$ as the image of $Y \in \textbf{SH}^{\mathbb{C}}$ under this composition. Since $\mathcal{C}_q \xhookrightarrow{} \textbf{SH}^{\mathbb{C}}$, every cover comes with a natural map
\[f_q(Y) \to Y.\]
Since $\mathcal{C}_{q+1} \subset \mathcal{C}_q$, we also get natural maps
\[f_{q+1}(Y) \to f_q(Y).\]
These maps define a filtration of the object $Y$, so we also get a notion of slices by taking cofibers.
\end{Cons}

To provide some intuition, let us give an example of Construction \ref{Construction of filtrations using subcategories} in the classical stable homotopy category $\textbf{SH}^{\cl}$. Let $\mathcal{C}_q$ be the full subcategory of $\textbf{SH}^{\cl}$ generated under colimits and extensions by the set of spheres $\{S^n \mid n \geq q\}$. Then the $q$-th cover of $X \in \textbf{SH}^{\cl}$ defined in the above way is $\tau_{\geq q} X$, i.e. the $q$-th connective cover of $X$. The associated filtration of $X$ is sometimes called the Whitehead tower, or dual Postnikov tower, or just Postnikov tower.

Note how the connective cover in our example arose from restricting ourselves to a set of spheres with which we can build our spectra. In the motivic setting we have a bigraded set of spheres $S^{s, w}$, so there are different ways to restrict the indices. For example, we can restrict the weight $w$, or the coweight $s - w$, or both. Respectively, this defines the effective, connective, or very effective filtration. We provide a bit more detail:
\begin{itemize}
    \item For the effective filtration let $\mathcal{C}_q$ in Construction \ref{Construction of filtrations using subcategories} be the full subcategory of $\textbf{SH}^{\mathbb{C}}$ generated under colimits and extensions by the set
    \[\{S^{s, w} \mid w \geq q\}.\]
    An object $Y \in \mathcal{C}_q$ is called $q$-effective. The covers are typically denoted by $f_q(Y) \to Y$ and the slices by $s_q(Y)$.
    
    \item For the connective filtration let $\mathcal{C}_q$ in Construction \ref{Construction of filtrations using subcategories} be the full subcategory of $\textbf{SH}^{\mathbb{C}}$ generated under colimits and extensions by the set
    \[\{S^{s, w} \mid s - w \geq q\}.\]
    An object $Y \in \mathcal{C}_q$ is called $q$-connective. The covers are typically denoted by $Y_{\geq q} \to Y$ and the slices by $Y_{= q}$.
    
    \item For the very effective filtration let $\mathcal{C}_q$ in Construction \ref{Construction of filtrations using subcategories} be the full subcategory of $\textbf{SH}^{\mathbb{C}}$ generated under colimits and extensions by the set
    \[\{S^{s, w} \mid w \geq q \text{~and~} s - w \geq q\}.\]
    An object $Y \in \mathcal{C}_q$ is called $q$-very effective. The covers are typically denoted by $\tf_q(Y) \to Y$ and the slices by $\ts_q(Y)$.
\end{itemize}
Working with bigraded spheres can sometimes make it harder to see the underlying idea, so we provide a slight rephrasing of the above. Recall that in $S^{s, w}$ the coweight $s - w$ counts the number of topological (or simplicial) spheres and the weight $w$ counts the number of algebraic spheres. So $S^{a, 0} = S^a$ corresponds to a topological sphere of dimension $a$, and $S^{b, b}$ corresponds to a $b$-fold smash product of algebraic spheres $\mathbb{G}_m^b$. Then the above filtrations can also be understood as follows:
\begin{itemize}
    \item For the effective filtration we allow all topological spheres, but we restrict to algebraic spheres $\mathbb{G}_m^b$ with $b \geq q$.
    \item For the connective filtration we allow all algebraic spheres, but we restrict to topological spheres $S^a$ with $a \geq q$.
    \item For the very effective filtration we allow only topological and algebraic spheres of degree $\geq q$.
\end{itemize}

The following is a convenient characterization of covers defined using Construction \ref{Construction of filtrations using subcategories}. See also \cite[Lemma 2.7]{Hea19} and the surrounding discussion therein for another characterization.

\begin{Lemma}\phantomsection\label{Characterization of general covers defined using subcategories}
	Let $f_q(Y) \to Y$ be any cover defined using Construction \ref{Construction of filtrations using subcategories}. Then $f_q(Y) \to Y$ is characterized up to equivalence by the following properties:
	\begin{enumerate}[label=\normalfont{(\arabic*)}]
		\item $f_q(Y) \in \mathcal{C}_q$,
		\item for any $M \in \mathcal{C}_q$ the induced map $[M, f_q(Y)] \to [M, Y]$ is an isomorphism.
	\end{enumerate}
\end{Lemma}

\begin{proof}
	Yoneda lemma.
\end{proof}

In our setting, it suffices to check the second condition of Lemma \ref{Characterization of general covers defined using subcategories} on the generators of the category $\mathcal{C}_q$.

\begin{Cor}\phantomsection\label{Suffices to check iso on mapping spaces on generators}
	Let $f_q(Y) \to Y$ be any cover defined using Construction \ref{Construction of filtrations using subcategories}. Assume that $\mathcal{C}_q$ is generated under colimits and extensions by a set of objects $\mathcal{S}$. Then $f_q(Y) \to Y$ is characterized up to equivalence by the following properties:
	\begin{enumerate}[label=\normalfont{(\arabic*)}]
		\item $f_q(Y) \in \mathcal{C}_q$,
		\item for any $S \in \mathcal{S}$ the induced map $[S, f_q(Y)] \to [S, Y]$ is an isomorphism.
	\end{enumerate}
\end{Cor}

\begin{proof}
	We want to invoke Lemma \ref{Characterization of general covers defined using subcategories}. By assumption $[S, f_q(Y)] \to [S, Y]$ is an isomorphism for any $S \in \mathcal{S}$. Since $\mathcal{S}$ generates $\mathcal{C}_q$ under colimits and extensions, it thus suffices to show that the property of $[M, f_q(Y)] \to [M, Y]$ being an isomorphism is preserved under colimits and extensions. For colimits this follows from the natural isomorphism
	\[[\varinjlim_{i \in I} M_i, -] \cong \varprojlim_{i \in I} [M_i, -].\]
	For extensions it follows from the long exact sequence associated to a cofiber sequence and the five lemma.
\end{proof}

We end this section by describing a general method to prove that the entire class of $q$-effective, $q$-connective, or $q$-very effective objects satisfies some given property. This is a standard result, but we include it here for convenience since we will make use of it multiple times.

\begin{Lemma}\phantomsection\label{Lemma on properties of full subcategories generated by a set of objects}
    Let $\mathcal{C}$ be a stable $\infty$-category, and let $\mathcal{D} \subset \mathcal{C}$ be a full subcategory generated under colimits and extensions by a set of objects $\mathcal{S} \subset \mathcal{C}$. Let $P$ be a property that an object of $\mathcal{C}$ can satisfy. Then $P$ holds for all objects of $\mathcal{D}$ if:
    \begin{itemize}
        \item $P$ holds for every object of $\mathcal{S}$,
        \item $P$ is preserved under cofibers, meaning if $X \to Y \to Z$ is a cofiber sequence, and $P$ holds for $X$ and $Y$, then $P$ also holds for $Z$,
        \item $P$ is preserved under extensions, meaning if $X \to Y \to Z$ is a cofiber sequence, and $P$ holds for $X$ and $Z$, then $P$ also holds for $Y$.
        \item $P$ is preserved under filtered colimits, meaning if $\{X_i\}_{i \in I}$ is a filtered system, and $P$ holds for each $X_i$, then $P$ holds for $\varinjlim X_i$.
    \end{itemize}
\end{Lemma}

\begin{proof}
    Since $P$ holds for every object of $\mathcal{S}$, and since those objects generate $\mathcal{D}$ via colimits and extensions, it suffices to check that $P$ is preserved under colimits and extensions. For extensions that is true by assumption. For colimits, recall the general $\infty$-categorical fact \cite[Proposition 4.4.3.2]{HTT} that every colimit can be written in terms of coequalizers and coproducts. The coequalizer of two maps $f, g$ is equivalently the cofiber of $f - g$, so that coequalizers preserve $P$ again by assumption. To see that arbitrary coproducts preserve $P$, recall that they can be written as filtered colimits of their finite sub-coproducts \cite[Section 4.2.3]{HTT}. Filtered colimits preserve $P$ by assumption. Finite coproducts can be expressed as a finite sequence of extensions, so they also preserve $P$. Altogether, we see that $P$ is preserved under colimits and extensions and hence $P$ holds for all objects of $\mathcal{D}$.
\end{proof}

\begin{Cor}\phantomsection\label{Corollary for proving properties of filtrations}
    Let $P$ be a property that an object of ${\normalfont \textbf{SH}^{\mathbb{C}}}$ can satisfy. Assume $P$ is preserved under cofibers, extensions, and filtered colimits. Then:
    \begin{itemize}
        \item $P$ holds for all $q$-effective objects if it holds for the set
        \[\{S^{s, w} \mid w \geq q\},\]
        \item $P$ holds for all $q$-connective objects if it holds for the set
        \[\{S^{s, w} \mid s - w \geq q\},\]
        \item $P$ holds for all $q$-very effective objects if it holds for the set
        \[\{{\normalfont S^{s, w} \mid w \geq q \text{~and~} s - w \geq q}\}.\]
    \end{itemize}
\end{Cor}

\begin{proof}
    Apply Lemma \ref{Lemma on properties of full subcategories generated by a set of objects} to the full subcategories $\mathcal{C}_q \subset \textbf{SH}^{\mathbb{C}}$ used in the definitions of the terms $q$-effective, $q$-connective, or $q$-very effective.
\end{proof}

\begin{Rem}
	By Theorem \ref{Equivalence between C-motivic homotopy and filtered spectra} we can interpret the $\mathbb{C}$-motivic homotopy category as a category of modules of filtered spectra. We could extend the definitions of the filtrations in this section to the entire category $\textbf{SH}^{\fil}$ by replacing the motivic spheres with spheres in filtered spectra. Note that under the equivalence of Theorem \ref{Equivalence between C-motivic homotopy and filtered spectra} the motivic sphere spectrum corresponds to $\Gamma_\star(S)$, which is not a sphere spectrum in $\textbf{SH}^{\fil}$. So using the evident similar looking definitions for the effective, connective, or very effective filtrations with filtered sphere spectra instead of motivic sphere spectra would yield different filtrations, a priori. However, by \cite[Lemma 2.15]{Hea19}, these filtered sphere versions do in fact induce the motivic sphere filtrations on the category of $\Gamma_\star(S)$-modules. For this, one needs to check that $\Gamma_\star(S)$ is effective, connective, or very effective when defining these terms using filtered spheres. It suffices to check that $\Gamma_\star(S)$ is very effective, and that in turn follows by cellular approximation in filtered spectra. For cellular approximation in the $\mathbb{C}$-motivic stable homotopy category see Section \ref{section: cellular approximation in C-motivic stable homotopy theory}. For filtered spectra similar ideas work.
\end{Rem}

\section{Homotopy characterizations of \texorpdfstring{$\mathbb{C}$}{C}-motivic covers}\label{section: Homotopy characterizations of C-motivic covers}

In this section we will determine characterizations of each of the three covers defined in Section \ref{section: C-motivic filtrations} via properties of their homotopy groups. For the classical connective covers such a characterization is well-known. We will give a short proof of this classical fact in Section \ref{section: Characterizing the classical connective cover}. This will help us understand the similarities and differences for the characterizations of our motivic covers. Specifically, the classical characterization makes use of cellular approximation (also sometimes called CW-approximation). On the level of spaces one can always form a cellular approximation by starting with a set of $0$-cells, then attaching $1$-cells, $2$-cells, and so on to build a space with the correct homotopy groups. Using this method for a spectrum requires that it is bounded below, so that one has a place to start the inductive process. The starting point corresponds to the lowest degree non-trivial homotopy group of the spectrum to be approximated. In the motivic setting, homotopy groups are bigraded. To structure an induction we need to put some restrictions on the possible shape of the non-trivial homotopy groups in the $(s, w)$-plane. Classically, the bounded below assumption, i.e. restricting $s$, is enough. Motivically, just restricting $s$ is not enough, as a spectrum could still have arbitrary complexity in the $w$-direction, so that there is still no evident place to start an induction. In Section \ref{section: cellular approximation in C-motivic stable homotopy theory} we will give a sufficient condition for the homotopy groups of a $\mathbb{C}$-motivic spectrum $Y$ so that cellular approximation is possible.

As it turns out, the homotopy groups of a $\mathbb{C}$-motivic very effective cover always take on a form such that cellular approximation is possible. However, for effective and connective covers that is not the case. This isn't surprising, considering the effective and connective filtrations each only restrict one (linear combination) of the two values $s$ and $w$, whereas the very effective filtration restricts two. For the connective filtration one can work around this issue by writing it as the colimit of its effective covers, and those covers can then be built using cellular approximation. For the effective filtration we can put an additional restriction on the spectrum to be covered so that cellular approximation can be used. We will see later that this restriction is satisfied for objects of the form $\Gamma_*(X)$ where $X \in \textbf{SH}^{\cl}$ is bounded below. So, in our case of interest, we can characterize every cover by some properties of its homotopy groups.

We should also mention that there is a homotopy sheaf characterization of the motivic connective filtration that holds over any field \cite[Theorem 2.3]{Hoy15}. Nevertheless, we give a proof in our $\mathbb{C}$-motivic cellular setting here for completeness' sake.

The first subsection will recall the characterization of classical connective covers in terms of their homotopy groups. The second subsection will describe how we can structure a cellular approximation method in the $\mathbb{C}$-motivic stable homotopy category. In the remaining subsections we will characterize motivic covers in terms of their homotopy groups.

\subsection{Characterizing the classical connective cover}\label{section: Characterizing the classical connective cover}

The following Lemma \ref{characterization of classical q-th connective cover} is a well-known characterization of the classical $q$-connective cover. We provide a short proof, because slight adjustments to the arguments will eventually lead to similar characterizations of our $\mathbb{C}$-motivic covers.

Recall that we say a classical spectrum is $q$-connective if it is contained in the full subcategory $\mathcal{C}_q$ of $\textbf{SH}^{\cl}$ generated under colimits and extensions by $\{S^n \mid n \geq q\}$. The $q$-th connective cover of a spectrum $X$ was then defined via the composition
\[\textbf{SH}^{\cl} \to \mathcal{C}_q \xhookrightarrow{} \textbf{SH}^{\cl}\]
where the second map is the inclusion and the first map is its right adjoint.

\begin{Lemma}\phantomsection\label{Homotopy groups of classical q-connective spectra}
	Let $X \in {\normalfont \textbf{SH}^{\cl}}$. Then the following are equivalent:
	\begin{enumerate}[label=\normalfont{(\arabic*)}]
		\item $X$ is $q$-connective,
		\item $\pi_s(X) \cong 0$ for $s < q$.
	\end{enumerate}
\end{Lemma}

\begin{proof}
	We start with the implication from (1) to (2). Apply Lemma \ref{Lemma on properties of full subcategories generated by a set of objects} to the property $\pi_s(X) \cong 0$ for $s < q$. This property is obviously satisfied by the spheres generating the $q$-connective subcategory. That the property is preserved under cofibers and extensions follows from the long exact sequence of homotopy groups induced by a cofiber sequence. That it is preserved under filtered colimits follows from the compactness of the sphere spectrum. It follows that every $q$-connective object satisfies property (2).
	
	For the reverse direction, note that property (2) says that $X$ is bounded below. As such, it can be built using cellular approximation. Due to the structure of the homotopy groups of $X$, this process does not need cells of degree less than $q$. This implies $X$ is $q$-connective.
\end{proof}

\begin{Lemma}\phantomsection\label{characterization of classical q-th connective cover}
	Let $X \in {\normalfont \textbf{SH}^{\cl}}$ and let $\tau_{\geq q}(X) \to X$ denote the $q$-th connective cover of $X$. Then this datum is characterized up to equivalence by the following properties:
	\begin{enumerate}[label=\normalfont{(\arabic*)}]
		\item The spectrum $\tau_{\geq q}(X)$ is $q$-connective,
		\item the induced map $\pi_s(\tau_{\geq q}(X)) \to \pi_s(X)$ is an isomorphism for $s \geq q$.
	\end{enumerate}
\end{Lemma}

\begin{proof}
	Follows immediately from a classical analogue of Corollary \ref{Suffices to check iso on mapping spaces on generators}.
\end{proof}

Given that property (1) is equivalent to a property on the homotopy groups of $X$, one can also check that some map $Z \to X$ constitutes a $q$-th connective cover purely in terms of homotopy groups.

\begin{Prop}\phantomsection\label{characterization of classical q-th connective cover v2}
	Let $Z \to X$ in ${\normalfont \textbf{SH}^{\cl}}$ satisfy the following properties:
	\begin{enumerate}[label=\normalfont{(\arabic*)}]
		\item The induced map $\pi_s(Z) \to \pi_s(X)$ is an isomorphism for $s \geq q$,
		\item we have $\pi_s(Z) \cong 0$ for $s < q$.
	\end{enumerate}
	Then $Z \to X$ exhibits $Z$ as the $q$-th connective cover of $X$.
\end{Prop}

\begin{proof}
	By Lemma \ref{Homotopy groups of classical q-connective spectra} property (2) here is equivalent to property (1) in Lemma \ref{characterization of classical q-th connective cover}. The claim follows.
\end{proof}

\subsection{Cellular approximation in \texorpdfstring{$\mathbb{C}$}{C}-motivic stable homotopy theory}\label{section: cellular approximation in C-motivic stable homotopy theory}

Classically, to approximate a spectrum $X$ by a cellular one up to equivalence, we only need to assume that $X$ is bounded below. Then a cellular approximation is constructed inductively by attaching cells to form the correct homotopy groups. This happens in the realm of abelian groups because homotopy groups are modules over $\pi_0(S) = \mathbb{Z}$.

In the $\mathbb{C}$-motivic case we observe some differences: 
\begin{enumerate}
	\item We have two gradings, so we need to explain what type of boundedness condition we want,
	\item homotopy groups are modules over $\pi_{0, *}(S) = \mathbb{Z}[\tau]$, so attaching a cell changes more than just one weight in homotopy,
	\item we need to think about how to structure our induction, again because of the bigrading.
\end{enumerate}
The third point will shape our answers to the other two, so we consider it first. Classically, when constructing a cellular approximation of some spectrum, we start with the lowest stem because the homotopy groups of the sphere spectrum $S$ are concentrated in stems $s \geq 0$. So after attaching cells to form the correct lowest stem we can start changing the next stem while keeping the lowest stem unchanged. Now looking at the $\mathbb{C}$-motivic sphere we have $\pi_{s, w}(S) \cong 0$ when $s < 0$ or when $s - w < 0$. Both of these facts are immediate from the model $\Gamma_\star(S)$ for the $\mathbb{C}$-motivic sphere spectrum. We provide a schematic picture of the homotopy groups of $S$ in the $(s, w)$-plane to help visualize these conditions. The area shaded in gray is where the homotopy groups of the sphere spectrum are potentially non-trivial.

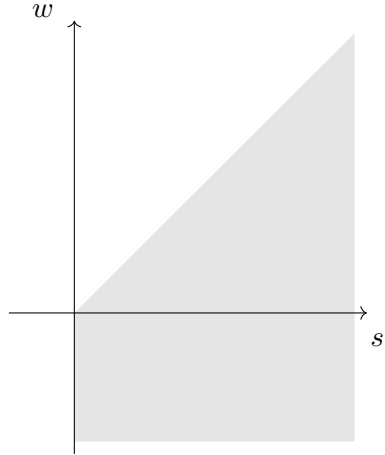
\begin{center}
	\begin{tikzpicture}
		\node (origin) at (0, 0) {};
		\node (top right) at (3.7, 3.7) {};
		\node (bottom right) at (3.7, -1.7) {};
		\node (bottom left) at (0, -1.7) {};
		
		\fill[gray, opacity = 0.2] (origin.center) -- (top right.center) -- (bottom right.center) -- (bottom left.center) -- cycle;
		
		\node (s0) at (-1, 0) {};
		\node[label=below:{$s$}] (s1) at (4, 0) {};
		\node (w0) at (0, -2) {};
		\node[label=left:{$w$}] (w1) at (0, 4) {};
		
		\draw[->] (s0) to (s1);
		\draw[->] (w0) to (w1);
	\end{tikzpicture}
\end{center}

The vertical cutoff is the condition $s < 0$ and the diagonal cutoff is the condition $s - w < 0$. In particular, in each given stem there is an upper bound to the weights which can have non-trivial homotopy groups. Based on these observations, we can structure our induction as follows:
\begin{itemize}
	\item Start by moving from small stems to large stems,
	\item inside each stem go from large weights to small weights.
\end{itemize}
Notice also that $\tau$ preserves the stem and decreases the weight by one, so by moving downwards in weights we can build the homotopy groups in each stem as $\mathbb{Z}[\tau]$-modules by just attaching cells for $\mathbb{Z}[\tau]$-generators and relations.

In conclusion, if we want to use cellular approximation for a $\mathbb{C}$-motivic spectrum $Y$ then we need to assume that its homotopy groups satisfy the following properties.
\begin{Def}\phantomsection\label{Definition of approximable}
	We say that $Y \in \textbf{SH}^{\mathbb{C}}$ is \textit{approximable} if the following two conditions are satisfied:
	\begin{enumerate}[label=\normalfont{(\arabic*)}]
		\item $\pi_{s, w}(Y)$ is bounded below in stems $s$,
		\item for each stem $s$, $\pi_{s, w}(Y)$ is bounded above in weights $w$.
	\end{enumerate}
\end{Def}
The second condition is for example satisfied if $Y$ is bounded below in coweights. In turn, this is satisfied if $Y$ is $r$-connective for some $r \in \mathbb{Z}$, as we will see in Proposition \ref{homotopy groups of q-connective spectra vanish in low coweights}.

The following Proposition \ref{motivic cellular approximation} will be integral for proving the homotopy group characterizations of the motivic covers considered in this article. Note that the conditions (1), (2) and (3) therein are satisfied by $q$-effective, $q$-connective, and $q$-very effective spectra, respectively, as we will see in Propositions \ref{tau-iso on homotopy groups of q-effective spectra}, \ref{homotopy groups of q-connective spectra vanish in low coweights}, and \ref{Homotopy characterization of q-very effective spectra}. So Proposition \ref{motivic cellular approximation} states that those properties are sufficient characterizations, assuming $Y$ is approximable. As it turns out, if $Y$ satisfies condition (3) then it is always approximable. We will prove that in Proposition \ref{Very effective spectra are always approximable} below.

\begin{Prop}\phantomsection\label{motivic cellular approximation}
	Let $Y \in {\normalfont \textbf{SH}^{\mathbb{C}}}$ be approximable, and let $q \in \mathbb{Z}$. Then there exists a spectrum $Z \in {\normalfont \textbf{SH}^{\mathbb{C}}}$ together with an equivalence $Z \to Y$ such that:
	\begin{enumerate}[label=\normalfont{(\arabic*)}]
		\item If the map
		\[\tau: \pi_{s, w}(Y) \to \pi_{s, w - 1}(Y)\]
		is an isomorphism for all $w \leq q$ then $Z$ is $q$-effective,
		\item if $\pi_{s, w}(Y) \cong 0$ for $s - w < q$ then $Z$ is $q$-connective,
		\item If the map
		\[\tau: \pi_{s, w}(Y) \to \pi_{s, w - 1}(Y)\]
		is an isomorphism for all $w \leq q$ and $\pi_{s, w}(Y) \cong 0$ for $s - w < q$ then $Z$ is $q$-very effective.
	\end{enumerate}
	It follows in the respective cases that $Y$ is also $q$-effective, $q$-connective, or $q$-very effective, as these properties are closed under equivalences.
\end{Prop}

\begin{proof}
	We start by giving a general construction of a cellular approximation $Z \to Y$ inspired by cellular approximation in the classical setting. That $Z$ has the claimed properties given the respective hypotheses will be a consequence of the construction.
	
	By induction, we will build a compatible system of spectra $Z_{m, n}$ for $m, n \in \mathbb{Z}$ together with maps $Z_{m, n} \to Y$ such that
	\begin{equation}
		\pi_{s, w}(Z_{m, n}) \to \pi_{s, w}(Y) \text{ is an isomorphism if either $s = m$ and $w \geq n$, or $s < m$.} \tag{$\star$} \label{inductive hypothesis}
	\end{equation}
	Taking the colimit as $m \to \infty$ and $n \to -\infty$ will yield a map $Z \to Y$ with the desired properties.
	
	Since $Y$ is assumed approximable, there exists a smallest value $a$ such that $\pi_{a, *}(Y) \neq 0$. Moreover, there exists a largest value $b$ such that $\pi_{a, b}(Y) \neq 0$. So we can let $Z_{a, b + 1}$ be trivial and the unique map $Z_{a, b + 1} \to Y$ satisfies the desired property \eqref{inductive hypothesis}.
	
	Next we describe the inductive step to get from $Z_{a, n}$ to $Z_{a, n - 1}$. Assume we have $Z_{a, n} \to Y$ satisfying \eqref{inductive hypothesis}. Just as in the classical setting, attach a wedge of spheres $S^{a, n - 1}$ in order to produce a map that induces a surjection in $\pi_{a, n - 1}$. Then attach further cells to make the kernel trivial. That the spectrum constructed in this way comes with a map satisfying \eqref{inductive hypothesis} follows just as in the classical construction. Let $Z_{a, -\infty}$ denote the colimit over the $Z_{a, n}$ as $n \to -\infty$. By construction $Z_{a, -\infty}$ satisfies \eqref{inductive hypothesis}.
	
	Then continue with stem $s = a + 1$. There is a largest weight $c$ such that $\pi_{a + 1, c}(Y) \neq 0$ since $Y$ is approximable. Note however that we may need to start attaching cells in an even larger weight since the cells we attached for $Z_{a, -\infty}$ can produce non-trivial homotopy groups in larger weights. But since $\pi_{a, w}(Y) = 0$ for $w > b$, the largest possible weight where we may need to start attaching cells is given by $\max\{b + 1, c\}$. That is because the homotopy groups of the $\mathbb{C}$-motivic sphere $S$ have a vanishing line of slope $1$ in the $(s, w)$-plane. In particular $\max\{b + 1, c\}$ is a finite number, so that we can repeat the inductive process as before to build $Z_{a + 1, -\infty}$. Iterating this procedure we can then form $Z_{m, -\infty}$ satisfying \eqref{inductive hypothesis} for every $m \in \mathbb{Z}$. Let $Z = Z_{\infty, -\infty}$ denote the colimit as $m \to \infty$. By construction we have a map $Z \to Y$ inducing an isomorphism on all homotopy groups, making it the desired equivalence.
	
	Now we consider the different hypotheses:
	\begin{enumerate}[label=\normalfont{(\arabic*)}]
		\item If there is a $q \in \mathbb{Z}$ such that $\tau: \pi_{s, w}(Y) \to \pi_{s, w - 1}(Y)$ is an isomorphism for all $w \leq q$, then the inductive process to build $Z_{m, n}$ will not need spheres of weights less than $q$ because $Z_{m, q}$ already has the correct homotopy groups in those weights by construction. To explain why, recall that in the homotopy groups of the sphere $\tau: \pi_{s, w}(S) \to \pi_{s, w - 1}(S)$ is an isomorphism for $w \leq 0$. We claim that $\tau: \pi_{s, w}(Z_{a, n - 1}) \to \pi_{s, w - 1}(Z_{a, n - 1})$ is an isomorphism for $w \leq n - 1$. We can inductively assume it is true for $Z_{a, n}$. Then $Z_{a, n - 1}$ is constructed by attaching spheres $S^{a, n - 1}$ to $Z_{a, n}$, and $\tau$ acts as an isomorphism on the homotopy groups of $S^{a, n - 1}$ starting in weight $n - 1$. The long exact sequences of homotopy groups we get when attaching cells, together with the naturality of $\tau$, and the five lemma imply the claim. In particular, this means when we form $Z_{a, q}$ we already have an isomorphism $\pi_{a, *}(Z_{a, q}) \to \pi_{a, *}(Y)$ as this map is an isomorphism for weights $w \geq q$ by construction, and for $w \leq q$ it is then also an isomorphism by naturality of $\tau$ and the fact that $\tau$ acts as an isomorphism on both sides. So we can set $Z_{a, -\infty} = Z_{a, q}$. Inductively repeating this argument shows that $\tau$ will always be an isomorphism in weights $w \leq q$ for every $Z_{m, n}$. So the approximation $Z$ that we are building does not need cells of weight less than $q$. It follows that $Z$ is $q$-effective.
		
		\item By assumption we have $\pi_{s, w}(Y) \cong 0$ if $s - w < q$ for some $q \in \mathbb{Z}$. Since the homotopy groups of the sphere vanish in coweights $s - w < 0$, it follows that the inductive process does not need spheres of coweights $s - w < q$. That is because $Z_{m, n}$ will also satisfy $\pi_{s, w}(Z_{m, n}) \cong 0$ if $s - w < q$ by induction. In consequence, $Z$ is $q$-connective.
		
		\item From the proofs of the other two cases, it follows that the approximation $Z$ does not need cells of weights $w < q$ or of coweights $s - w < q$. Therefore, $Z$ is $q$-very effective.
	\end{enumerate}
\end{proof}

\begin{Prop}\phantomsection\label{Very effective spectra are always approximable}
	Let $Y \in {\normalfont \textbf{SH}^{\mathbb{C}}}$ and let $q \in \mathbb{Z}$. Assume $Y$ satisfies the following properties:
	\begin{enumerate}[label=\normalfont{(\arabic*)}]
		\item The map
		\[\tau: \pi_{s, w}(Y) \to \pi_{s, w - 1}(Y)\]
		is an isomorphism for all $w \leq q$,
		\item we have $\pi_{s, w}(Y) \cong 0$ for $s - w < q$.
	\end{enumerate}
	Then $Y$ is approximable.
\end{Prop}

\begin{proof}
	To show that $Y$ is approximable it suffices to prove that the homotopy groups of $Y$ are bounded below in stems and in coweights. By assumption they are bounded below in coweights. We claim that properties (1) and (2) together imply that the homotopy groups $\pi_{s, w}(Y)$ are trivial for $s < 2q$, regardless of $w$. To see this, note that for such values of $s$ we have $\pi_{s, q}(Y) \cong 0$ by property (2) since $s - q < 2q - q = q$. The $\tau$-isomorphism property (1) then implies that $\pi_{s, w}(Y) \cong 0$ for $s < 2q$ and any $w \in \mathbb{Z}$. So $Y$ is approximable.
\end{proof}

\begin{Rem}\phantomsection\label{Larger zero region in homotopy groups of C-motivic very effective cover}
	We provide a visualization of the structure of the homotopy groups of $Y$ in Proposition \ref{Very effective spectra are always approximable}. In the proof we saw that the homotopy groups of $Y$ are trivial in all stems $s < 2q$. This means the potentially non-trivial homotopy groups of $Y$ are contained in the gray region depicted below. The sketch provided here should remind the reader of the one at the beginning of Section \ref{section: cellular approximation in C-motivic stable homotopy theory}. It is the same shape, but the origin is shifted to the point $(2q, q)$.
	
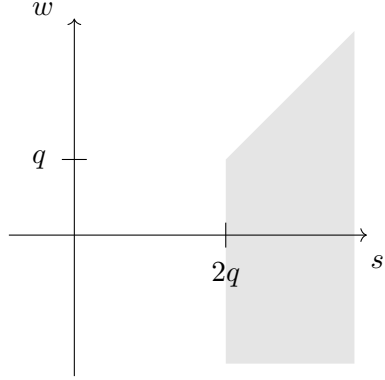
	\begin{center}
		\begin{tikzpicture}
			\node (shaded area upper left) at (2, 1) {};
			\node (shaded area lower left) at (2, -1.7) {};
			\node (shaded area upper right) at (3.7, 2.7) {};
			\node (shaded area lower right) at (3.7, -1.7) {};
			
			\fill[gray, opacity = 0.2] (shaded area upper left.center) -- (shaded area lower left.center) -- (shaded area lower right.center) -- (shaded area upper right.center) -- cycle;
			
			\node (s0) at (-1, 0) {};
			\node[label=below:{$s$}] (s1) at (4, 0) {};
			\node (w0) at (0, -2) {};
			\node[label=left:{$w$}] (w1) at (0, 3) {};
			
			\draw[->] (s0) to (s1);
			\draw[->] (w0) to (w1);
			
			\node[label={[label distance = -.2cm]left:{$q$}}] (w-axis mark left) at (-.3, 1) {};
			\node (w-axis mark right) at (.3, 1) {};
			\node (s-axis mark q upper) at (1, .3) {};
			\node[label={[label distance = -.2cm]below:{$2q$}}] (s-axis mark 2q lower) at (2, -.3) {};
			\node (s-axis mark 2q upper) at (2, .3) {};
			
			\draw[-] (w-axis mark left) to (w-axis mark right);
			\draw[-] (s-axis mark 2q lower) to (s-axis mark 2q upper);
			
		\end{tikzpicture}
	\end{center}
	We will see later in Proposition \ref{Homotopy characterization of q-very effective spectra} that this is also the shape of the homotopy groups of any $\mathbb{C}$-motivic $q$-very effective spectrum because such spectra satisfy the hypotheses of Proposition \ref{Very effective spectra are always approximable}. This shape is specific to $\mathbb{C}$-motivic homotopy theory. For example, over $\mathbb{R}$ we do not have $\tau$ as a map on homotopy groups, but instead the sphere spectrum has a self-map $\rho$ with degree $(s, w) = (-1, -1)$ giving rise to non-trivial homotopy groups in arbitrarily small stems.

\end{Rem}

\subsection{Characterizing the effective filtration}

In this section, we will give a characterization of effective covers via properties of their homotopy groups. As was discussed in the introduction to Section \ref{section: Homotopy characterizations of C-motivic covers}, we cannot give a characterization purely in terms of homotopy groups for effective covers of every spectrum, as there is no evident starting point for the inductive process of cellular approximation. However, we will see that restricting to spectra $Y$ that are $r$-connective for some $r \in \mathbb{Z}$ will imply that cellular approximation is possible for the cover of $Y$, so that we can describe the cover purely via properties of its homotopy groups.

Let us first take note of a general property that the homotopy groups of all $q$-effective spectra satisfy.

\begin{Prop}\phantomsection\label{tau-iso on homotopy groups of q-effective spectra}
	Let $Y \in {\normalfont \textbf{SH}^{\mathbb{C}}}$ be $q$-effective. Then the map
	\[\tau: \pi_{s, w}(Y) \to \pi_{s,w - 1}(Y)\]
	is an isomorphism for all $w \leq q$.
\end{Prop}

\begin{proof}
	By Corollary \ref{Corollary for proving properties of filtrations} it suffices to prove that $\tau$ is an isomorphism in the given range for the set of $\mathbb{C}$-motivic spheres $\{S^{s, w} \mid w \geq q\}$, and that this property is preserved under cofibers, extensions, and filtered colimits.
	
	To see that $\tau$ is an isomorphism for these spheres we can consider the model $\Gamma_\star(S)$ for the motivic sphere spectrum. Then $\tau$ is an isomorphism
	\[\tau: \pi_{s, w}(\Gamma_\star(S)) \to \pi_{s, w - 1}(\Gamma_\star(S))\]
	for all $w \leq 0$ because there are no elements in the ANSS for the classical sphere spectrum on negative-indexed antidiagonals. Shifting this result implies the claim for the set $\{S^{s, w} \mid w \geq q\}$.
	
	Now we consider the different preservation conditions. For cofibers and extensions the result is immediate from the long exact sequence of homotopy groups induced by a cofiber sequence and the five lemma. For filtered colimits it follows from the fact that spheres are compact objects in the motivic stable homotopy category, see \cite[Theorem 9.1]{DI05}, meaning homotopy groups commute with filtered colimits.
\end{proof}

\begin{Rem}
	It is known that the effective slice spectral sequence does not converge for all spectra. That is because there are spectra that are $q$-effective for all $q \in \mathbb{Z}$. In \cite[Remark 2.1]{Voe02a} Voevodsky gives the example
	\[M\mathbb{Z}/p[\tau^{-1}] = \varinjlim(M\mathbb{Z}/p \xrightarrow[]{\tau} M\mathbb{Z}/p \xrightarrow[]{\tau} \dotsc)\]
	of a spectrum that is $q$-effective for all $q \in \mathbb{Z}$, where $p$ is any prime not equal to the characteristic of the field one is working over. It follows from Proposition \ref{tau-iso on homotopy groups of q-effective spectra} that any $\mathbb{C}$-motivic spectrum that is $q$-effective for all $q \in \mathbb{Z}$ must have homotopy groups where $\tau$ is an isomorphism in all degrees.
\end{Rem}

Next, we provide a characterization of the covers making up the $\mathbb{C}$-motivic effective filtration. This characterization works for any spectrum $Y \in \textbf{SH}^{\mathbb{C}}$, but it requires the assumption that the $q$-th effective cover $f_q(Y)$ is $q$-effective. We will provide another characterization afterwards that replaces the $q$-effective assumption on the cover by an assumption on $Y$ itself.

\begin{Lemma}\phantomsection\label{characterization of q-th effective cover}
	Let $Y \in {\normalfont \textbf{SH}^{\mathbb{C}}}$ and let $f_q(Y) \to Y$ denote the $q$-th effective cover of $Y$. Then this datum is characterized up to equivalence by the following properties:
	\begin{enumerate}[label=\normalfont{(\arabic*)}]
		\item The spectrum $f_q(Y)$ is $q$-effective,
		\item the induced map $\pi_{s, w}(f_q(Y)) \to \pi_{s, w}(Y)$ is an isomorphism for $w \geq q$.
	\end{enumerate}
\end{Lemma}

\begin{proof}
	This is an immediate consequence of Corollary \ref{Suffices to check iso on mapping spaces on generators} and the definition of the $q$-effective subcategory.
\end{proof}

If we want to use Lemma \ref{characterization of q-th effective cover} to see that some datum $Z \to Y$ constitutes a $q$-effective cover, then we will have to show that $Z$ is $q$-effective. Under a reasonable hypothesis on $Y$ we can guarantee that a spectrum $Z$ satisfying property (2) of Lemma \ref{characterization of q-th effective cover} as well as the property of Proposition \ref{tau-iso on homotopy groups of q-effective spectra} is $q$-effective. Our specific use case will be where $Y = \Gamma_\star(X)$ with $X \in \textbf{SH}^{\cl}$ bounded below.

\begin{Prop}\phantomsection\label{characterization of q-th effective cover for nice spectra}
	Let $Z \to Y$ in ${\normalfont \textbf{SH}^{\mathbb{C}}}$ satisfy the following properties:
	\begin{enumerate}[label=\normalfont{(\arabic*)}]
		\item The induced map $\pi_{s, w}(Z) \to \pi_{s, w}(Y)$ is an isomorphism for $w \geq q$,
		\item the map
		\[\tau: \pi_{s, w}(Z) \to \pi_{s,w - 1}(Z)\]
		is an isomorphism for all $w \leq q$.
	\end{enumerate}
	If $Y$ is $r$-connective for some $r \in \mathbb{Z}$, then $Z \to Y$ exhibits $Z$ as the $q$-th effective cover of $Y$.
\end{Prop}

\begin{proof}
	Given Lemma \ref{characterization of q-th effective cover} it suffices to show that $Z$ is $q$-effective. This will follow from Proposition \ref{motivic cellular approximation} once we have shown that $Z$ is approximable as in Definition \ref{Definition of approximable}. For that, it suffices to show that $\pi_{s, w}(Z)$ is bounded below in stems $s$ and in coweights $s - w$.
	
	First we claim that $\pi_{s, w}(Z) \cong 0$ for any $s < q + r$. Because $Y$ is assumed $r$-connective, we have $\pi_{s, w}(Y) \cong 0$ for $s - w < r$ by Proposition \ref{homotopy groups of q-connective spectra vanish in low coweights}\footnote{We prove Proposition \ref{homotopy groups of q-connective spectra vanish in low coweights} in the next section to preserve the usual ordering of effective, connective, and very effective. The proof of Proposition \ref{homotopy groups of q-connective spectra vanish in low coweights} is independent of this one, so there is no circularity.}. Equivalently $\pi_{s, w}(Y) \cong 0$ for $s < w + r$. By property (1) this implies $\pi_{s, w}(Z) \cong 0$ for $s < w + r$ and $w \geq q$. In particular when $w = q$ we get $\pi_{s, q}(Z) \cong 0$ for $s < q + r$. Then property (2) implies that it is also true for $s < q + r$ and $w \leq q$. Together with property (1) that means $\pi_{s, w}(Z) \cong 0$ for $s < q + r$ and any $w \in \mathbb{Z}$.
	
	Lastly we claim that the homotopy groups of $Z$ are trivial in coweights $s - w < r$: If $w \geq q$ then this follows from the fact that it is true for $Y$ as we saw above. If $w \leq q$ then note that $s - w < r$ implies $s < q + r$. But we saw that the homotopy groups of $Z$ are trivial for such values of $s$. Altogether we have shown that $Z$ is approximable and therefore $q$-effective by Proposition \ref{motivic cellular approximation}.
\end{proof}

\subsection{Characterizing the connective filtration}

In this section, we will give a characterization of $\mathbb{C}$-motivic connective covers via properties of their homotopy groups. As was mentioned in the introduction to Section \ref{section: Homotopy characterizations of C-motivic covers}, this characterization holds more generally in the $k$-motivic stable homotopy category for any field $k$, see \cite[Theorem 2.3]{Hoy15}.

Let us first take note of a general property that the homotopy groups of all $q$-connective spectra satisfy.

\begin{Prop}\phantomsection\label{homotopy groups of q-connective spectra vanish in low coweights}
	Let $Y \in {\normalfont \textbf{SH}^{\mathbb{C}}}$ be $q$-connective. Then we have $\pi_{s, w}(Y) \cong 0$ for $s - w < q$.
\end{Prop}

\begin{proof}
	By Corollary \ref{Corollary for proving properties of filtrations} it suffices to prove that the homotopy groups vanish in the given region for the set of objects $\{S^{a, b} \mid a - b \geq q\}$, and that this property is preserved under cofibers, extensions, and filtered colimits.
	
	To see that the homotopy groups vanish for these spheres in coweights less than $q$, first consider the model $\Gamma_\star(S)$ for the $\mathbb{C}$-motivic sphere spectrum. The ANSS for the classical sphere has the main diagonal $s = f$ as a vanishing line, meaning the ANSS is trivial above that line. Recall that in the BKSS for $\Gamma_w(S)$ every degree below the $2w$-th antidiagonal given by $s + f = 2w$ is trivial. The intersection point of these two lines is at $(w, w)$. In particular, for $s < w$ we have $\pi_s(\Gamma_w(S)) = 0$ because any point $(s, f)$ with $s < w$ is either above the main diagonal, below the $2w$-th antidiagonal, or both. The inequality $s < w$ can be rewritten as $s - w < 0$, so we have shown that the homotopy groups of the motivic sphere spectrum are trivial in negative coweights. Then $\pi_{s, w}(S^{a, b}) = \pi_{s - a, w - b}(S)$, so if $a - b \geq q$ and $s - w < q$, then
	\[s - a - (w - b) = s - w - (a - b) < q - q = 0,\]
	which implies $\pi_{s, w}(S^{a, b}) = \pi_{s - a, w - b}(S) = \pi_{s - a}(\Gamma_{w - b}(S)) = 0$. This shows that the set of objects $\{S^{a, b} \mid a - b \geq q\}$ satisfies the necessary property.
	
	It remains to check that the homotopy group vanishing condition is preserved under cofibers, extensions, and filtered colimits. For cofibers and extensions it follows from the long exact sequence of homotopy groups induced by a cofiber sequence. For filtered colimits it follows because spheres are compact objects, see \cite[Theorem 9.1]{DI05}, meaning homotopy groups commute with filtered colimits.
\end{proof}

Next, we provide a characterization of the covers making up the $\mathbb{C}$-motivic connective filtration. This characterization uses the assumption that the $q$-th connective cover $f_q(Y)$ is $q$-connective. We will provide another characterization afterwards that replaces the $q$-connective assumption on the cover by a vanishing property for the homotopy groups of the cover in certain coweights.

\begin{Lemma}\phantomsection\label{characterization of motivic q-th connective cover}
	Let $Y \in {\normalfont \textbf{SH}^{\mathbb{C}}}$ and let $Y_{\geq q} \to Y$ denote the $q$-th connective cover of $Y$. Then this datum is characterized up to equivalence by the following properties:
	\begin{enumerate}[label=\normalfont{(\arabic*)}]
		\item The spectrum $f_q(Y)$ is $q$-connective,
		\item the induced map $\pi_{s, w}(Y_{\geq q}) \to \pi_{s, w}(Y)$ is an isomorphism for $s - w \geq q$.
	\end{enumerate}
\end{Lemma}

\begin{proof}
	This is an immediate consequence of Corollary \ref{Suffices to check iso on mapping spaces on generators} and the definition of the $q$-connective subcategory.
\end{proof}

If we want to use Lemma \ref{characterization of motivic q-th connective cover} to see that some datum $Z \to Y$ constitutes a $q$-connective cover, then we will have to show that $Z$ is $q$-connective. This can equivalently be expressed in terms of a condition on the homotopy groups of $Z$.

\begin{Prop}\phantomsection\label{characterization of motivic q-th connective cover v2}
	Let $Z \to Y$ in ${\normalfont \textbf{SH}^{\mathbb{C}}}$ satisfy the following properties:
	\begin{enumerate}[label=\normalfont{(\arabic*)}]
		\item The induced map $\pi_{s, w}(Z) \to \pi_{s, w}(Y)$ is an isomorphism for $s - w \geq q$,
		\item we have $\pi_{s, w}(Z) \cong 0$ for $s - w < q$.
	\end{enumerate}
	Then $Z \to Y$ exhibits $Z$ as the $q$-th connective cover of $Y$.
\end{Prop}

\begin{proof}
	Given Lemma \ref{characterization of motivic q-th connective cover} it suffices to show that $Z$ is $q$-connective. We can use effective covers to write $Z$ as the colimit of $f_r(Z)$ as $r \to -\infty$. Then for $Z$ to be $q$-connective it suffices to show that $f_r(Z)$ is $q$-connective for all $r \in \mathbb{Z}$ since the $q$-connective subcategory is by definition closed under colimits.
	
	By Proposition \ref{motivic cellular approximation} $q$-connectivity of $f_r(Z)$ follows once we have shown
	\begin{enumerate}[label=\normalfont{(\alph*)}]
		\item $f_r(Z)$ is approximable,
		\item we have $\pi_{s, w}(f_r(Z)) \cong 0$ if $s - w < q$.
	\end{enumerate}
	Note that by definition and by Proposition \ref{tau-iso on homotopy groups of q-effective spectra} the cover $f_r(Z) \to Z$ satisfies the two properties
	\begin{enumerate}[label=\normalfont{(\roman*)}]
		\item $\pi_{s, w}(f_r(Z)) \to \pi_{s, w}(Z)$ is an isomorphism for $w \geq r$,
		\item $\tau: \pi_{s, w}(f_r(Z)) \to \pi_{s, w - 1}(f_r(Z))$ is an isomorphism for $w \leq r$.
	\end{enumerate}
	Properties (2) and (i) imply $\pi_{s, w}(f_r(Z)) \cong 0$ if $s - w < q$ and $w \geq r$. In particular, this means $\pi_{s, r}(f_r(Z)) \cong 0$ if $s - r < q$, or equivalently if $s < r + q$. Then property (ii) further implies $\pi_{s, w}(f_r(Z)) \cong 0$ whenever $s < r + q$ and $w \leq r$. We can summarize these vanishing conditions as $\pi_{s, w}(f_r(Z)) \cong 0$ if $s - w < q$ or $s < r + q$. In particular this implies properties (a) and (b), i.e. that $f_r(Z)$ is approximable and that its homotopy groups vanish in coweights less than $q$. So by Proposition \ref{motivic cellular approximation} every cover $f_r(Z)$ is $q$-connective, and therefore the colimit $Z$ is too.
\end{proof}

\begin{Rem}
	Comparing Proposition \ref{characterization of motivic q-th connective cover v2} to Proposition \ref{characterization of q-th effective cover for nice spectra}, the reader might notice an asymmetry in the assumptions: For the effective cover, we need $Y$ to be $r$-connective, but for the connective cover we do not need any assumption on $Y$. This asymmetry also appears in the compositions of the different covers. We have
	\[f_q \circ (-)_{\geq q} \simeq \tf_q \not\simeq (-)_{\geq q} \circ f_q.\]
	For the first equivalence, which holds over perfect fields $k$, see \cite[Lemma 10]{Bac17}. For the second non-equivalence one can see this for example by using the homotopy group characterizations provided in this article.
\end{Rem}

\subsection{Characterizing the very effective filtration}

In this section, we will give a characterization of $\mathbb{C}$-motivic very effective covers via properties of their homotopy groups. 

It turns out that being $q$-very effective can be characterized purely in terms of homotopy groups. That is because the structure of the non-trivial homotopy groups of a $q$-very effective spectrum $Y$ always implies that $Y$ is approximable.

\begin{Prop}\phantomsection\label{Homotopy characterization of q-very effective spectra}
	Let $Y \in {\normalfont \textbf{SH}^{\mathbb{C}}}$. Then $Y$ is $q$-very effective if and only if the following conditions are satisfied:
	\begin{enumerate}[label=\normalfont{(\arabic*)}]
		\item The map
		\[\tau: \pi_{s, w}(Y) \to \pi_{s,w - 1}(Y)\]
		is an isomorphism for all $w \leq q$,
		\item we have $\pi_{s, w}(Y) \cong 0$ for $s - w < q$.
	\end{enumerate}
\end{Prop}

\begin{proof}
	Assume first that $Y$ is $q$-very effective. This implies $Y$ is $q$-effective and $q$-connective since the $q$-very effective subcategory is generated by a subset of the generators for the $q$-effective and $q$-connective subcategories. Then $Y$ satisfies properties (1) and (2) by Proposition \ref{tau-iso on homotopy groups of q-effective spectra} and Proposition \ref{homotopy groups of q-connective spectra vanish in low coweights}, respectively.
	
	For the reverse direction, we want to use Proposition \ref{motivic cellular approximation}. It thus suffices to show that $Y$ is approximable. That follows from Proposition \ref{Very effective spectra are always approximable}.
\end{proof}

Next, we provide two characterizations of the covers making up the $\mathbb{C}$-motivic very effective filtration. The first will use the assumption that the covering spectrum is $q$-very effective. The second characterization will replace this assumption by properties of the homotopy groups of the cover to be.

\begin{Lemma}\phantomsection\label{characterization of q-th very effective cover}
	Let $Y \in {\normalfont \textbf{SH}^{\mathbb{C}}}$ and let $\tf_q(Y) \to Y$ denote the $q$-th very effective cover of $Y$. Then this datum is characterized up to equivalence by the following properties:
	\begin{enumerate}[label=\normalfont{(\arabic*)}]
		\item The spectrum $\tf_q(Y)$ is $q$-very effective,
		\item the induced map $\pi_{s, w}(\tf_q(Y)) \to \pi_{s, w}(Y)$ is an isomorphism if both $w \geq q$ and $s - w \geq q$.
	\end{enumerate}
\end{Lemma}

\begin{proof}
	This is an immediate consequence of Corollary \ref{Suffices to check iso on mapping spaces on generators} and the definition of the $q$-very effective subcategory.
\end{proof}

\begin{Prop}\phantomsection\label{characterization of q-th very effective cover v2}
	Let $Z \to Y$ in ${\normalfont \textbf{SH}^{\mathbb{C}}}$ satisfy the following properties:
	\begin{enumerate}[label=\normalfont{(\arabic*)}]
		\item The induced map $\pi_{s, w}(Z) \to \pi_{s, w}(Y)$ is an isomorphism if both $w \geq q$ and $s - w \geq q$,
		\item the map
		\[\tau: \pi_{s, w}(Z) \to \pi_{s,w - 1}(Z)\]
		is an isomorphism for all $w \leq q$,
		\item we have $\pi_{s, w}(Z) \cong 0$ for $s - w < q$.
	\end{enumerate}
	Then $Z \to Y$ exhibits $Z$ as the $q$-th very effective cover of $Y$.
\end{Prop}

\begin{proof}
	By Proposition \ref{Homotopy characterization of q-very effective spectra} properties (2) and (3) here are equivalent to property (1) in Lemma \ref{characterization of q-th very effective cover}. The claim follows.
\end{proof}

\section{Filtered spectrum models for covers of \texorpdfstring{$\mathbb{C}$}{C}-motivic analogues}\label{section: Filtered spectrum models for covers of C-motivic analogues}

In this section, we give filtered spectrum models for the effective, connective and very effective filtrations in $\mathbb{C}$-motivic homotopy theory in the case where our spectrum $Y \in \textbf{SH}^\mathbb{C}$ is of the form $Y = \Gamma_\star(X)$ for a spectrum $X \in \textbf{SH}^{\cl}$ that is bounded below. Here, $\Gamma_\star$ is the $\mathbb{C}$-motivic analogue functor of \cite{GIKR}, as recalled in Section \ref{section: The filtered spectrum model for the C-motivic stable homotopy category}. Examples of spectra $Y = \Gamma_\star(X)$ include $S$, $\kq$, $\mmf$, $\kgl$, and $\MGL$, where $X$ is $S$, $\ko$, $\tmf$, $\ku$, and $\MU$, respectively.

All of our models will be defined by making slight changes to the definition of $\Gamma_\star$ itself. The notation used for these modified versions is meant to be reminiscent of the notation for the different motivic covers as in Section \ref{section: C-motivic filtrations}. In order to prove that our new filtered spectra defined in this way do in fact model the respective covers we will use the characterizations via homotopy groups spelled out in Section \ref{section: Homotopy characterizations of C-motivic covers}. Since we define our spectra in a similar way to $\Gamma_\star$, they will come with Bousfield-Kan spectral sequences computing their homotopy groups. Comparing those spectral sequences with the BKSS for $\Gamma_\star(X)$ itself will make it easy to show that our models satisfy the respective homotopy group characterizations.

In this section, we will ignore the difference between a filtered spectrum and its corresponding $\mathbb{C}$-motivic spectrum under the equivalence of Theorem \ref{Equivalence between C-motivic homotopy and filtered spectra}. By this we mean that we will think of $\Gamma_\star(X)$ as both a filtered spectrum with a BKSS, and a $\mathbb{C}$-motivic spectrum such that the notions of filtrations and covers as discussed in the previous sections can be applied.

We start by making the following observation about $\Gamma_\star(X)$ when the input spectrum $X \in \textbf{SH}^{\cl}$ is bounded below.

\begin{Prop}\phantomsection\label{Motivic analogues of bounded below spectra are very effective}
    Let $X \in {\normalfont \textbf{SH}^\cl}$ be $2q$-connective so that $\pi_s(X) \cong 0$ for $s < 2q$. Then $\Gamma_\star(X)$ is $q$-very effective.
\end{Prop}

\begin{proof}
    We will use the characterization of $q$-very effective given in Proposition \ref{Homotopy characterization of q-very effective spectra}. For property (1), meaning that $\tau$ is an isomorphism on homotopy groups in weights $w \leq q$, consider the BKSS for $\Gamma_\star(X)$. The  ANSS for a $0$-connective spectrum has the main diagonal as a vanishing line, meaning it is trivial to the left and above this line. So the ANSS for a $2q$-connective spectrum will have a diagonal vanishing line $s - 2q = f$ through the point $(2q, 0)$. In weights $w \leq q$ this implies that $\tau$ will be an isomorphism, as the truncation below the $2w$-th antidiagonal has the same trivial effect for all $w \leq q$.
    
    Lastly we need to check property (2) of Proposition \ref{Homotopy characterization of q-very effective spectra} that the homotopy groups of $\Gamma_\star(X)$ are trivial in coweights $s - w < q$. In the ANSS for $X$ we have the diagonal vanishing line $s - 2q = f$ from above. In weight $w$ the BKSS has the additional vanishing line given by the $2w$-th antidiagonal $s + f = 2w$. The intersection point of these two lines is $(w + q, w - q)$. In particular, in weight $w$ the homotopy groups will be trivial in stems $s < w + q$. Rewriting this as $s - w < q$ shows that $\Gamma_\star(X)$ satisfies property (2).
\end{proof}

An alternative approach to the above proof would be to prove first that the motivic analogue of a connective spectrum is very effective and then shift the result by using the suspension functor together with \cite[Lemma 3.13]{GIKR}.

\subsection{Filtered spectrum models for effective covers}\label{section: filtered spectrum models for effective covers}

Let $X \in \textbf{SH}^{\cl}$ be bounded below. For $q \in \mathbb{Z}$ define a filtered spectrum $\Gamma_\star^q(X)$ by letting the degree $w$-part be
\begin{align*}
\Gamma_w^q(X) &=
\begin{dcases}
\Tot(\tau_{\geq 2w} (X \wedge \MU^{\wedge \bullet + 1})), & w \geq q\\
\Tot(\tau_{\geq 2q} (X \wedge \MU^{\wedge \bullet + 1})), & w \leq q
\end{dcases}\\[.5em]
 &=
\begin{dcases}
\Gamma_w(X), & w \geq q\\
\Gamma_q(X), & w \leq q.
\end{dcases}
\end{align*}
The maps $\Gamma_w^q(X) \to \Gamma_{w-1}^q(X)$ are induced by $\tau_{\geq 2w} \to \tau_{\geq 2w - 2}$ in filtrations $w > q$, and they are the identity otherwise. Note that there is a map
\[\Gamma_\star^q(X) \to \Gamma_\star(X)\]
given by the identity in filtrations $w \geq q$ and induced by $\tau_{\geq 2q} \to \tau_{\geq 2w}$ for $w \leq q$.

\begin{Lemma}\phantomsection\label{Gamma_star^q(X) is a Gamma_star(S)-module}
    The filtered spectrum $\Gamma_\star^q(X)$ is a $\Gamma_\star(S)$-module.
\end{Lemma}

\begin{proof}
    Recall from \cite[Proposition 3.8]{GIKR} that $\Gamma_\star(X)$ has a $\Gamma_\star(S)$-module structure. The module structure for $\Gamma_\star^q(X)$ is derived from this one via a careful case distinction.
    
    We need a map
    \[\Gamma_\star(S) \wedge \Gamma_\star^q(X) \to \Gamma_\star^q(X).\]
    In weight $w$ we have by definition
    \[(\Gamma_\star(S) \wedge \Gamma_\star^q(X))_w =\varinjlim_{i + j \geq w} \Gamma_i(S) \wedge \Gamma_j^q(X).\]
    So we need to define a compatible system of maps
    \[\Gamma_i(S) \wedge \Gamma_j^q(X) \to \Gamma_w^q(X)\]
    for all $i + j \geq w$. To do this, we need to distinguish between different cases for the values $j$, $w$, and $q$. Compatibility of the maps will follow by construction.
    
    If $j \geq q$ and $w \geq q$, then we have
    \[\Gamma_i(S) \wedge \Gamma_j(X) \to \Gamma_w(X),\]
    and we can use the same map as in the $\Gamma_\star(S)$-module structure of $\Gamma_\star(X)$.

    If $j \leq q$ and $w \geq q$, then we need
    \[\Gamma_i(S) \wedge \Gamma_q(X) \to \Gamma_w(X).\]
    Since $i + q \geq i + j \geq w$, we can again use a map from the $\Gamma_\star(S)$-module structure of $\Gamma_\star(X)$ for this.
    
    If $j \geq q$ and $w \leq q$, then we need a map
    \[\Gamma_i(S) \wedge \Gamma_j(X) \to \Gamma_q(X).\]
    If additionally $i \geq 0$, then $i + j \geq j \geq q$, so we can reuse the map from the $\Gamma_\star(S)$-module structure of $\Gamma_\star(X)$. If $i \leq 0$, then we have $\Gamma_i(S) \simeq S \simeq \Gamma_0(S)$ because the ANSS for the sphere converges and is trivial on negative-indexed antidiagonals. So we can form the composite
    \[\Gamma_i(S) \wedge \Gamma_j(X) \simeq \Gamma_0(S) \wedge \Gamma_j(X) \to \Gamma_q(X),\]
    where the last map comes from the $\Gamma_\star(S)$-module structure of $\Gamma_\star(X)$, which exists because $j \geq q$.

    If $j \leq q$ and $w \leq q$, then we need
    \[\Gamma_i(S) \wedge \Gamma_q(X) \to \Gamma_q(X).\]
    If $i \geq 0$, then $i + q \geq q$, so we can reuse the map from the $\Gamma_\star(S)$-module structure of $\Gamma_\star(X)$. If $i \leq 0$, then we have $\Gamma_i(S) \simeq S \simeq \Gamma_0(S)$, so that we can form the composite
    \[\Gamma_i(S) \wedge \Gamma_q(X) \simeq \Gamma_0(S) \wedge \Gamma_q(X) \to \Gamma_q(X),\]
    where the last map again comes from the $\Gamma_\star(S)$-module structure of $\Gamma_\star(X)$.
\end{proof}

\begin{Rem}
    The $\Gamma_\star(S)$-module structure for $\Gamma_\star^q(X)$ defined in Lemma \ref{Gamma_star^q(X) is a Gamma_star(S)-module} is compatible with changes in $q$. In particular, upon taking the colimit $q \to -\infty$ we recover the $\Gamma_\star(S)$-module structure of $\Gamma_\star(X)$ because then only the case $w \geq q$ from the proof of Lemma \ref{Gamma_star^q(X) is a Gamma_star(S)-module} is relevant, and in that case we used the same maps as in the $\Gamma_\star(S)$-module structure of $\Gamma_\star(X)$.
\end{Rem}

Because of Lemma \ref{Gamma_star^q(X) is a Gamma_star(S)-module} the filtered spectrum $\Gamma_\star^q(X)$ corresponds to a $\mathbb{C}$-motivic object under the equivalence of Theorem \ref{Equivalence between C-motivic homotopy and filtered spectra}. By abuse of notation, $\Gamma_\star^q(X)$ will denote both the filtered spectrum and the corresponding $\mathbb{C}$-motivic object.

The object $\Gamma_\star^q(X)$ comes with a sequence of BKSSs computing its homotopy groups. We provide a schematic picture of a part of this sequence

\noindent\resizebox{\textwidth}{!}
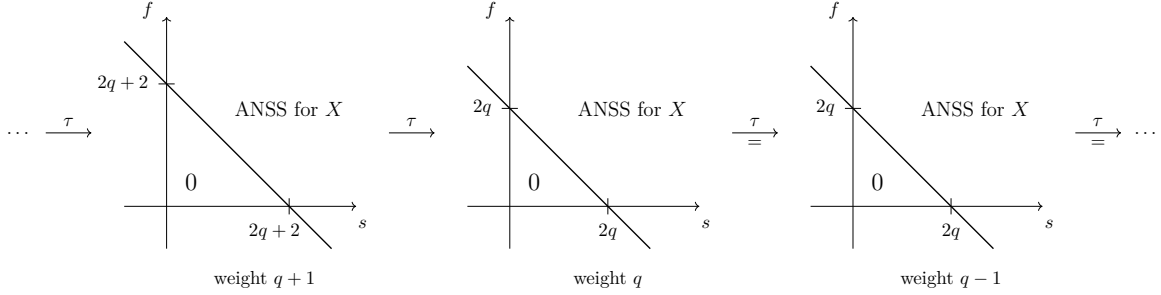
{
\begin{tikzpicture}
    \node (dots left) at (4 - 7, 1.5) {$\dotsc$};
    
    \node (arrow left) at (4.4 - 7, 1.5) {};
    \node (arrow right) at (5.6 - 7, 1.5) {};

    \draw[->] (arrow left) -- (arrow right) node[midway,above] {$\tau$};
    
    \node (s0) at (-1, 0) {};
    \node[label=below:{$s$}] (s1) at (4, 0) {};
    \node (f0) at (0, -1) {};
    \node[label=left:{$f$}] (f1) at (0, 4) {};
    
    \draw[->] (s0) to (s1);
    \draw[->] (f0) to (f1);
    
    \node (upper ad) at (-1, 3 + .5) {};
    \node (lower ad) at (3 + .5, -1) {};

    \draw[-, thick] (upper ad) to (lower ad);

    \node[label={[label distance = -.2cm]left:{$2q + 2$}}] (f-axis mark left) at (-.3, 2 + .5) {};
    \node (f-axis mark right) at (.3, 2 + .5) {};
    \node[label={[label distance = -.2cm]below:{$2q + 2\hspace{.6cm}$}}] (s-axis mark lower) at (2 + .5, -.3) {};
    \node (s-axis mark upper) at (2 + .5, .3) {};

    \draw[-] (f-axis mark left) to (f-axis mark right);
    \draw[-] (s-axis mark lower) to (s-axis mark upper);
    
    \node (zero) at (0.5, 0.5) {\Large $0$};
    \node (ANSS) at (2.5, 2) {\large ANSS for $X$};
    \node (weight) at (2, -1.5) {weight $q + 1$};

    \node (middle arrow left) at (4.4, 1.5) {};
    \node (middle arrow right) at (5.6, 1.5) {};

    \draw[->] (middle arrow left) -- (middle arrow right) node[midway,above] {$\tau$};

    \node (s0) at (-1 + 7, 0) {};
    \node[label=below:{$s$}] (s1) at (4 + 7, 0) {};
    \node (f0) at (0 + 7, -1) {};
    \node[label=left:{$f$}] (f1) at (0 + 7, 4) {};
    
    \draw[->] (s0) to (s1);
    \draw[->] (f0) to (f1);
    
    \node (upper ad) at (-1 + 7, 3) {};
    \node (lower ad) at (3 + 7, -1) {};

    \draw[-, thick] (upper ad) to (lower ad);

    \node[label={[label distance = -.2cm]left:{$2q$}}] (f-axis mark left) at (-.3 + 7, 2) {};
    \node (f-axis mark right) at (.3 + 7, 2) {};
    \node[label={[label distance = -.2cm]below:{$2q$}}] (s-axis mark lower) at (2 + 7, -.3) {};
    \node (s-axis mark upper) at (2 + 7, .3) {};

    \draw[-] (f-axis mark left) to (f-axis mark right);
    \draw[-] (s-axis mark lower) to (s-axis mark upper);
    
    \node (zero) at (0.5 + 7, 0.5) {\Large $0$};
    \node (ANSS) at (2.5 + 7, 2) {\large ANSS for $X$};
    \node (weight) at (2 + 7, -1.5) {weight $q$};

    \node (right arrow left) at (4.4 + 7, 1.5) {};
    \node (right arrow right) at (5.6 + 7, 1.5) {};

    \draw[->] (right arrow left) -- (right arrow right) node[midway,above] {$\tau$} node[midway,below] {$=$};

    \node (s0) at (-1 + 14, 0) {};
    \node[label=below:{$s$}] (s1) at (4 + 14, 0) {};
    \node (f0) at (0 + 14, -1) {};
    \node[label=left:{$f$}] (f1) at (0 + 14, 4) {};
    
    \draw[->] (s0) to (s1);
    \draw[->] (f0) to (f1);
    
    \node (upper ad) at (-1 + 14, 3) {};
    \node (lower ad) at (3 + 14, -1) {};

    \draw[-, thick] (upper ad) to (lower ad);

    \node[label={[label distance = -.2cm]left:{$2q$}}] (f-axis mark left) at (-.3 + 14, 2) {};
    \node (f-axis mark right) at (.3 + 14, 2) {};
    \node[label={[label distance = -.2cm]below:{$2q$}}] (s-axis mark lower) at (2 + 14, -.3) {};
    \node (s-axis mark upper) at (2 + 14, .3) {};

    \draw[-] (f-axis mark left) to (f-axis mark right);
    \draw[-] (s-axis mark lower) to (s-axis mark upper);
    
    \node (zero) at (0.5 + 14, 0.5) {\Large $0$};
    \node (ANSS) at (2.5 + 14, 2) {\large ANSS for $X$};
    \node (weight) at (2 + 14, -1.5) {weight $q - 1$};

    \node (right arrow left) at (4.4 + 14, 1.5) {};
    \node (right arrow right) at (5.6 + 14, 1.5) {};

    \draw[->] (right arrow left) -- (right arrow right) node[midway,above] {$\tau$} node[midway,below] {$=$};

    \node (dots right) at (6 + 14, 1.5) {$\dotsc$};
\end{tikzpicture}
}

Recall that the filtration degrees in $\Gamma_\star(X)$ correspond to the motivic weights. The BKSS for $\Gamma_\star^q(X)$ agrees with the one for $\Gamma_\star(X)$ in weights $w \geq q$. In weights $w \leq q$ we "freeze" the picture. This has the effect of making $\tau$ the identity on homotopy groups. Based on this observation we can prove the following Proposition.

\begin{Prop}\phantomsection\label{Filtered spectrum model for q-th effective cover}
    The datum $\Gamma_\star^q(X) \to \Gamma_\star(X)$ exhibits $\Gamma_\star^q(X)$ as the $q$-th effective cover of $\Gamma_\star(X)$.
\end{Prop}

\begin{proof}
    We want to use Proposition \ref{characterization of q-th effective cover for nice spectra}. So we need to check that $\Gamma_\star^q(X) \to \Gamma_\star(X)$ satisfies properties (1) and (2) therein, and that $\Gamma_\star(X)$ is $r$-connective for some $r \in \mathbb{Z}$. Properties (1) and (2) are immediate from the BKSS as discussed above. Regarding $r$-connectivity of $\Gamma_\star(X)$, recall that we are assuming that $X$ is bounded below throughout this section. So Proposition \ref{Motivic analogues of bounded below spectra are very effective} implies $\Gamma_\star(X)$ is $r$-very effective for some $r \in \mathbb{Z}$. This in turn implies it is $r$-connective because the $r$-very effective subcategory is generated by a subset of the $r$-connective one. So $\Gamma_\star^q(X) \to \Gamma_\star(X)$ satisfies all necessary conditions of Proposition \ref{characterization of q-th effective cover for nice spectra}, making it the $q$-th effective cover.
\end{proof}

\subsection{Filtered spectrum models for connective covers}\label{section: filtered spectrum models for connective covers}

Let $X \in \textbf{SH}^{\cl}$ be bounded below. For $q \in \mathbb{Z}$ define a filtered spectrum $\Gamma_\star^{\geq q}(X)$ by letting the degree $w$-part be
\begin{align*}
    \Gamma_w^{\geq q}(X) &= \tau_{\geq w + q}(\Tot(\tau_{\geq 2w}(X \wedge \MU^{\wedge \bullet + 1})))\\
     &= \tau_{\geq w + q}(\Gamma_w(X)).
\end{align*}
The maps $\Gamma_w^{\geq q}(X) \to \Gamma_{w - 1}^{\geq q}(X)$ are induced by $\tau_{\geq w + q} \to \tau_{\geq w - 1 + q}$ and $\Gamma_w(X) \to \Gamma_{w - 1}(X)$. Note that there is a map
\[\Gamma_\star^{\geq q}(X) \to \Gamma_\star(X)\]
induced in filtration $w$ by $\tau_{\geq w + q} \to \tau_{\geq -\infty} = \text{id}$.

\begin{Lemma}\phantomsection\label{Gamma_star^geq q(X) is a Gamma_star(S)-module}
    The filtered spectrum $\Gamma_\star^{\geq q}(X)$ is a $\Gamma_\star(S)$-module.
\end{Lemma}

\begin{proof}
    We need a map
    \[\Gamma_\star(S) \wedge \Gamma_\star^{\geq q}(X) \to \Gamma_\star^{\geq q}(X).\]
    In weight $w$ we have by definition
    \[(\Gamma_\star(S) \wedge \Gamma_\star^{\geq q}(X))_w =\varinjlim_{i + j \geq w} \Gamma_i(S) \wedge \Gamma_j^{\geq q}(X).\]
    So we need to define maps
    \[\Gamma_i(S) \wedge \Gamma_j^{\geq q}(X) \to \Gamma_w^{\geq q}(X)\]
    for all $i + j \geq w$. Using the definition of $\Gamma_\star^{\geq q}$, we can rewrite this as
    \[\Gamma_i(S) \wedge \tau_{\geq j + q} \Gamma_j(X) \to \tau_{\geq w + q} \Gamma_w(X).\]
    From the BKSS for $\Gamma_i(S)$ it follows that it is $i$-connective. Then $i + j + q \geq w + q$ implies that we have an isomorphism
    \[[\Gamma_i(S) \wedge \tau_{\geq j + q} \Gamma_j(X), \tau_{\geq w + q} \Gamma_w(X)] \xrightarrow[]{\cong} [\Gamma_i(S) \wedge \tau_{\geq j + q} \Gamma_j(X), \Gamma_w(X)]\]
    given by post-composition with the map $\tau_{\geq w + q} \Gamma_w(X) \to \tau_{\geq -\infty} \Gamma_w(X) \simeq \Gamma_w(X)$. So it suffices to produce a map into $\Gamma_w(X)$. We obtain this desired map as the composite
    \[\Gamma_i(S) \wedge \tau_{\geq j + q} \Gamma_j(X) \to \Gamma_i(S) \wedge \Gamma_j(X) \to \Gamma_w(X),\]
    where the first map is the identity smashed with $\tau_{\geq j + q} \Gamma_j(X) \to \tau_{\geq -\infty} \Gamma_j(X) \simeq \Gamma_j(X)$, and the second map is the one coming from the $\Gamma_\star(S)$-module structure of $\Gamma_\star(X)$.
\end{proof}

\begin{Rem}\phantomsection\label{Gamma_star(S)-module structure for q-connective covers recovers structure of Gamma_star(X)}
    The $\Gamma_\star(S)$-module structure for $\Gamma_\star^{\geq q}(X)$ defined in Lemma \ref{Gamma_star^geq q(X) is a Gamma_star(S)-module} is compatible with changes in $q$. In particular, upon taking the colimit $q \to -\infty$ we recover the $\Gamma_\star(S)$-module structure of $\Gamma_\star(X)$ because then $\tau_{\geq w + q}$ becomes $\tau_{\geq -\infty} = \text{id}$.
\end{Rem}

Because of Lemma \ref{Gamma_star^geq q(X) is a Gamma_star(S)-module} the filtered spectrum $\Gamma_\star^{\geq q}(X)$ corresponds to a $\mathbb{C}$-motivic object under the equivalence of Theorem \ref{Equivalence between C-motivic homotopy and filtered spectra}. By abuse of notation, $\Gamma_\star^{\geq q}(X)$ will denote both the filtered spectrum and the corresponding $\mathbb{C}$-motivic object.

\begin{Prop}\phantomsection\label{Filtered spectrum model for q-th connective cover}
    The datum $\Gamma_\star^{\geq q}(X) \to \Gamma_\star(X)$ exhibits $\Gamma_\star^{\geq q}(X)$ as the $q$-th connective cover of $\Gamma_\star(X)$.
\end{Prop}

\begin{proof}
    We want to use Proposition \ref{characterization of motivic q-th connective cover v2}. So we need to check that $\Gamma_\star^{\geq q}(X) \to \Gamma_\star(X)$ satisfies properties (1) and (2) therein. By definition we have
    \[\pi_{s, w}(\Gamma_\star^{\geq q}(X)) \cong \begin{dcases}
        \pi_{s, w}(\Gamma_\star(X)), & s \geq w + q\\
        0, & s < w + q.
    \end{dcases}\]
    Taking $w$ to the other side of each inequality shows that $\Gamma_\star^{\geq q}(X) \to \Gamma_\star(X)$ satisfies properties (1) and (2) of Proposition \ref{characterization of motivic q-th connective cover v2}.
\end{proof}

\subsection{Filtered spectrum models for very effective covers}\label{section: filtered spectrum models for very effective covers}

Let $X \in \textbf{SH}^{\cl}$ be bounded below. For $q \in \mathbb{Z}$ define a filtered spectrum $\tGamma_\star^q(X)$ by letting the degree $w$-part be
\begin{align*}
\tGamma_w^q(X) &=
\begin{dcases}
\tau_{\geq w + q} (\Tot(\tau_{\geq 2w} (X \wedge \MU^{\wedge \bullet + 1}))), & w \geq q\\
\tau_{\geq 2q} (\Tot(\tau_{\geq 2q} (X \wedge \MU^{\wedge \bullet + 1}))), & w \leq q
\end{dcases}\\[.5em]
 &=
\begin{dcases}
\tau_{\geq w + q} (\Gamma_w(X)), & w \geq q\\
\tau_{\geq 2q} (\Gamma_q(X)), & w \leq q
\end{dcases}\\[.5em]
 &=
\begin{dcases}
\Gamma_w^{\geq q}(X), & w \geq q\\
\Gamma_q^{\geq q}(X), & w \leq q.
\end{dcases}
\end{align*}
The maps $\tGamma_w^q(X) \to \tGamma_{w-1}^q(X)$ are induced by $\tau_{\geq w + q} \to \tau_{\geq w - 1 + q}$ and $\Gamma_w(X) \to \Gamma_{w - 1}(X)$ in filtrations $w > q$, and the identity otherwise. Note that there is a map
\[\tGamma_\star^q(X) \to \Gamma_\star(X)\]
induced by $\tau_{\geq w + q} \to \tau_{\geq -\infty} = \text{id}$ in filtrations $w \geq q$ and induced by $\tau_{\geq 2q} \to \tau_{\geq -\infty} = \text{id}$ and $\Gamma_q(X) \to \Gamma_w(X)$ for $w \leq q$.

\begin{Rem}\phantomsection\label{very effective cover in small weights}
	The spectrum $\tau_{\geq 2q} (\Tot(\tau_{\geq 2q} (X \wedge \MU^{\wedge \bullet + 1}))) = \tau_{\geq 2q} (\Gamma_q(X)) = \Gamma_q^{\geq q}(X)$, which forms the weight $w \leq q$ part of $\tGamma_\star^q(X)$, is naturally equivalent to $\tau_{\geq 2q}(X)$. To see this, consider the BKSS for $\Tot(\tau_{\geq 2q} (X \wedge \MU^{\wedge \bullet + 1}))$. It is the ANSS for $X$ truncated below the $2q$-th antidiagonal. In stems $s \geq 2q$ this spectral sequence is exactly the same as the ANSS for $X$. Then applying the truncation $\tau_{\geq 2q}$ after running the BKSS yields a spectrum whose homotopy groups are naturally isomorphic to those of $X$ in stems $s \geq 2q$, and trivial in stems $s < 2q$. The claim follows.
\end{Rem}

\begin{Lemma}\phantomsection\label{tGamma_star^q(X) is a Gamma_star(S)-module}
    The filtered spectrum $\tGamma_\star^q(X)$ is a $\Gamma_\star(S)$-module.
\end{Lemma}

\begin{proof}
    Combine the ideas for the $q$-effective and $q$-connective covers from the proofs of Lemma \ref{Gamma_star^q(X) is a Gamma_star(S)-module} and Lemma \ref{Gamma_star^geq q(X) is a Gamma_star(S)-module}. In more detail, we need a map
    \[\Gamma_\star(S) \wedge \tGamma_\star^q(X) \to \tGamma_\star^q(X).\]
    In weight $w$ we have by definition
    \[(\Gamma_\star(S) \wedge \tGamma_\star^q(X))_w =\varinjlim_{i + j \geq w} \Gamma_i(S) \wedge \tGamma_j^q(X).\]
    So we need to define maps
    \[\Gamma_i(S) \wedge \tGamma_j^q(X) \to \tGamma_w^q(X)\]
    for all $i + j \geq w$.

    If $j \geq q$ and $w \geq q$, then this becomes
    \[\Gamma_i(S) \wedge \Gamma_j^{\geq q}(X) \to \Gamma_w^{\geq q}(X),\]
    and we can use the same map as in the $\Gamma_\star(S)$-module structure of $\Gamma_\star^{\geq q}(X)$.

    If $j \leq q$ and $w \geq q$, then we need
    \[\Gamma_i(S) \wedge \Gamma_q^{\geq q}(X) \to \Gamma_w^{\geq q}(X).\]
    Since $i + q \geq i + j \geq w$, we can again use a map from the $\Gamma_\star(S)$-module structure of $\Gamma_\star^{\geq q}(X)$.

    If $j \geq q$ and $w \leq q$, then we need a map
    \[\Gamma_i(S) \wedge \Gamma_j^{\geq q}(X) \to \Gamma_q^{\geq q}(X).\]
    If additionally $i \geq 0$, then $i + j \geq j \geq q$, so we can reuse the map from the $\Gamma_\star(S)$-module structure of $\Gamma_\star^{\geq q}(X)$. If $i \leq 0$, then we have $\Gamma_i(S) \simeq S \simeq \Gamma_0(S)$ because the ANSS for the sphere converges and is trivial on negative-indexed antidiagonals. So we can form the composite
    \[\Gamma_i(S) \wedge \Gamma_j^{\geq q}(X) \simeq \Gamma_0(S) \wedge \Gamma_j^{\geq q}(X) \to \Gamma_q^{\geq q}(X),\]
    where the last map comes from the $\Gamma_\star(S)$-module structure of $\Gamma_\star^{\geq q}(X)$, which exists because $j \geq q$.

    If $j \leq q$ and $w \leq q$, then we need
    \[\Gamma_i(S) \wedge \Gamma_q^{\geq q}(X) \to \Gamma_q^{\geq q}(X).\]
    If $i \geq 0$, then $i + q \geq q$, so we can reuse the map from the $\Gamma_\star(S)$-module structure of $\Gamma_\star^{\geq q}(X)$. If $i \leq 0$, then we have $\Gamma_i(S) \simeq S \simeq \Gamma_0(S)$, so that we can form the composite
    \[\Gamma_i(S) \wedge \Gamma_q^{\geq q}(X) \simeq \Gamma_0(S) \wedge \Gamma_q^{\geq q}(X) \to \Gamma_q^{\geq q}(X),\]
    where the last map comes from the $\Gamma_\star(S)$-module structure of $\Gamma_\star^{\geq q}(X)$.
\end{proof}

\begin{Rem}
    The $\Gamma_\star(S)$-module structure for $\tGamma_\star^q(X)$ defined in Lemma \ref{tGamma_star^q(X) is a Gamma_star(S)-module} is compatible with changes in $q$. In particular, upon taking the colimit $q \to -\infty$ we recover the $\Gamma_\star(S)$-module structure of $\Gamma_\star(X)$ because then only the case $w \geq q$ from the proof of Lemma \ref{tGamma_star^q(X) is a Gamma_star(S)-module} is relevant, and in that case we used the same maps as in the $\Gamma_\star(S)$-module structure of $\Gamma_\star^{\geq q}(X)$, which are compatible with changes in $q$ by Remark \ref{Gamma_star(S)-module structure for q-connective covers recovers structure of Gamma_star(X)}.
\end{Rem}

Because of Lemma \ref{tGamma_star^q(X) is a Gamma_star(S)-module} the filtered spectrum $\tGamma_\star^q(X)$ corresponds to a $\mathbb{C}$-motivic object under the equivalence of Theorem \ref{Equivalence between C-motivic homotopy and filtered spectra}. By abuse of notation, $\tGamma_\star^q(X)$ will denote both the filtered spectrum and the corresponding $\mathbb{C}$-motivic object.

To compute the homotopy groups of $\tGamma_\star^q(X)$, one first runs a BKSS similar to the one for the model $\Gamma_\star^q(X)$ for the $q$-th effective cover. We include a schematic picture here again

\noindent\resizebox{\textwidth}{!}
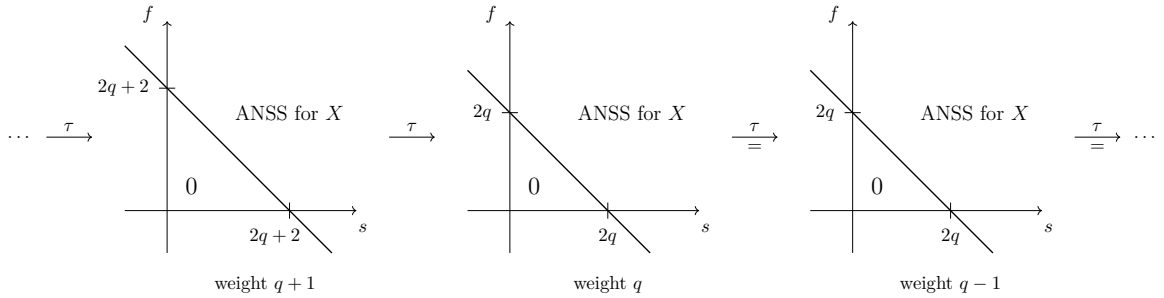
{
\begin{tikzpicture}
    \node (dots left) at (4 - 7, 1.5) {$\dotsc$};
    
    \node (arrow left) at (4.4 - 7, 1.5) {};
    \node (arrow right) at (5.6 - 7, 1.5) {};

    \draw[->] (arrow left) -- (arrow right) node[midway,above] {$\tau$};
    
    \node (s0) at (-1, 0) {};
    \node[label=below:{$s$}] (s1) at (4, 0) {};
    \node (f0) at (0, -1) {};
    \node[label=left:{$f$}] (f1) at (0, 4) {};
    
    \draw[->] (s0) to (s1);
    \draw[->] (f0) to (f1);
    
    \node (upper ad) at (-1, 3 + .5) {};
    \node (lower ad) at (3 + .5, -1) {};

    \draw[-, thick] (upper ad) to (lower ad);

    \node[label={[label distance = -.2cm]left:{$2q + 2$}}] (f-axis mark left) at (-.3, 2 + .5) {};
    \node (f-axis mark right) at (.3, 2 + .5) {};
    \node[label={[label distance = -.2cm]below:{$2q + 2\hspace{.6cm}$}}] (s-axis mark lower) at (2 + .5, -.3) {};
    \node (s-axis mark upper) at (2 + .5, .3) {};

    \draw[-] (f-axis mark left) to (f-axis mark right);
    \draw[-] (s-axis mark lower) to (s-axis mark upper);
    
    \node (zero) at (0.5, 0.5) {\Large $0$};
    \node (ANSS) at (2.5, 2) {\large ANSS for $X$};
    \node (weight) at (2, -1.5) {weight $q + 1$};

    \node (middle arrow left) at (4.4, 1.5) {};
    \node (middle arrow right) at (5.6, 1.5) {};

    \draw[->] (middle arrow left) -- (middle arrow right) node[midway,above] {$\tau$};

    \node (s0) at (-1 + 7, 0) {};
    \node[label=below:{$s$}] (s1) at (4 + 7, 0) {};
    \node (f0) at (0 + 7, -1) {};
    \node[label=left:{$f$}] (f1) at (0 + 7, 4) {};
    
    \draw[->] (s0) to (s1);
    \draw[->] (f0) to (f1);
    
    \node (upper ad) at (-1 + 7, 3) {};
    \node (lower ad) at (3 + 7, -1) {};

    \draw[-, thick] (upper ad) to (lower ad);

    \node[label={[label distance = -.2cm]left:{$2q$}}] (f-axis mark left) at (-.3 + 7, 2) {};
    \node (f-axis mark right) at (.3 + 7, 2) {};
    \node[label={[label distance = -.2cm]below:{$2q$}}] (s-axis mark lower) at (2 + 7, -.3) {};
    \node (s-axis mark upper) at (2 + 7, .3) {};

    \draw[-] (f-axis mark left) to (f-axis mark right);
    \draw[-] (s-axis mark lower) to (s-axis mark upper);
    
    \node (zero) at (0.5 + 7, 0.5) {\Large $0$};
    \node (ANSS) at (2.5 + 7, 2) {\large ANSS for $X$};
    \node (weight) at (2 + 7, -1.5) {weight $q$};

    \node (right arrow left) at (4.4 + 7, 1.5) {};
    \node (right arrow right) at (5.6 + 7, 1.5) {};

    \draw[->] (right arrow left) -- (right arrow right) node[midway,above] {$\tau$} node[midway,below] {$=$};

    \node (s0) at (-1 + 14, 0) {};
    \node[label=below:{$s$}] (s1) at (4 + 14, 0) {};
    \node (f0) at (0 + 14, -1) {};
    \node[label=left:{$f$}] (f1) at (0 + 14, 4) {};
    
    \draw[->] (s0) to (s1);
    \draw[->] (f0) to (f1);
    
    \node (upper ad) at (-1 + 14, 3) {};
    \node (lower ad) at (3 + 14, -1) {};

    \draw[-, thick] (upper ad) to (lower ad);

    \node[label={[label distance = -.2cm]left:{$2q$}}] (f-axis mark left) at (-.3 + 14, 2) {};
    \node (f-axis mark right) at (.3 + 14, 2) {};
    \node[label={[label distance = -.2cm]below:{$2q$}}] (s-axis mark lower) at (2 + 14, -.3) {};
    \node (s-axis mark upper) at (2 + 14, .3) {};

    \draw[-] (f-axis mark left) to (f-axis mark right);
    \draw[-] (s-axis mark lower) to (s-axis mark upper);
    
    \node (zero) at (0.5 + 14, 0.5) {\Large $0$};
    \node (ANSS) at (2.5 + 14, 2) {\large ANSS for $X$};
    \node (weight) at (2 + 14, -1.5) {weight $q - 1$};

    \node (right arrow left) at (4.4 + 14, 1.5) {};
    \node (right arrow right) at (5.6 + 14, 1.5) {};

    \draw[->] (right arrow left) -- (right arrow right) node[midway,above] {$\tau$} node[midway,below] {$=$};

    \node (dots right) at (6 + 14, 1.5) {$\dotsc$};
\end{tikzpicture}
}
Then, after running this BKSS, one has to take the respective connective cover depending on the weight $w$ to get the final answer for the homotopy groups of $\tGamma_\star^q(X)$. In essence, one removes all non-trivial classes from the $E_\infty$-page to the left of the vertical line $s = w + q$ or $s = 2q$, depending on the weight. In weights $w \leq q$ we also force $\tau$ to be the identity on homotopy groups again just as in our model for the $q$-th effective cover. Note that the outermost connected cover in the definition of $\tGamma_\star^q(X)$ doesn't affect this because for $w \leq q$ we take the same cover in every weight.

\begin{Prop}\phantomsection\label{Filtered spectrum model for q-th very effective cover}
    The datum $\tGamma_\star^q(X) \to \Gamma_\star(X)$ exhibits $\tGamma_\star^q(X)$ as the $q$-th very effective cover of $\Gamma_\star(X)$.
\end{Prop}

\begin{proof}
    We have to check that $\tGamma_\star^q(X) \to \Gamma_\star(X)$ satisfies the properties of Proposition \ref{characterization of q-th very effective cover v2}. Property (1) holds because in the case where $w \geq q$ the homotopy group $\pi_{s, w}(\tGamma_\star^q(X))$ is computed by running the same BKSS as for $\Gamma_\star(X)$, and then the truncation $\tau_{\geq w + q}$ has no effect because of $s - w \geq q$, i.e. $s \geq w + q$. The second property, that $\tau$ is an isomorphism on certain homotopy groups, is implied by the discussion on the BKSS just above this Proposition. Property (3), meaning that the homotopy groups of the cover vanish in coweights $s - w < q$, needs a case distinction: If $w \geq q$ then this property is implied by the truncation $\tau_{\geq w + q}$. If $w < q$ then we have $\pi_{s, w}(\tGamma_\star^q(X)) \cong  0$ when $s < 2q$. But $2q > w + q$, so in particular we have $\pi_{s, w}(\tGamma_\star^q(X)) \cong  0$ when $s < w + q$. Rearranging the inequality to $s - w < q$ yields the claim.
\end{proof}

\begin{Rem}
    By definition, the homotopy groups $\pi_{s,w}(\tGamma_\star^q(X))$ are trivial in all stems $s < 2q$ regardless of the weight. That is more than what is needed for property (3) of Proposition \ref{characterization of q-th very effective cover v2}. However, this should in fact be the case, as was discussed in Proposition \ref{Very effective spectra are always approximable} and Remark \ref{Larger zero region in homotopy groups of C-motivic very effective cover}.
\end{Rem}

\begin{Rem}
	By \cite[Lemma 7.3]{CQ21} complex Betti realization agrees with taking the colimit $w \to -\infty$ under the equivalence between motivic and filtered spectra, assuming that this colimit is attained for a finite value of $w$. It follows from Remark \ref{very effective cover in small weights} and Proposition \ref{Filtered spectrum model for q-th very effective cover} that the Betti realization of the $q$-th very effective cover of $\Gamma_\star(X)$ is $\tau_{\geq 2q}(X)$. This result is also implied by the claim in \cite[Section 3.3]{GRSO12}.
\end{Rem}

\begin{Rem}
    Over perfect fields $k$, \cite[Lemma 10]{Bac17} shows that there is an equivalence $\tf_q(Y) \simeq f_q (Y_{\geq q})$ for every motivic spectrum $Y$. Using our filtered spectrum methods, we could make the same observation if we phrased our constructions as endofunctors in filtered spectra. For example the effective cover $\Gamma_\star^q(X)$ can be obtained by applying to $\Gamma_\star(X)$ the functor that sends a filtered spectrum to the filtered spectrum that is itself in weights $w \geq q$ and constant for $w \leq q$. Similarly the connective cover $\Gamma_\star^{\geq q}(X)$ can be obtained from $\Gamma_\star(X)$ by composing with the functor that sends a filtered spectrum to the filtered spectrum where we apply $\tau_{\geq w + q}$ in weight $w$. The very effective cover $\tGamma_\star^q(X)$ can be defined by applying first the endofunctor for the connective cover to $\Gamma_\star(X)$, and then applying the effective cover functor to the resulting filtered spectrum. From this perspective, the equivalence $\tf_q(Y) \simeq f_q (Y_{\geq q})$ holds by definition for $Y = \Gamma_\star(X)$.
\end{Rem}

In \cite[Remark 2.5]{ARO20} it is stated that the effective and very effective covers of motivic Landweber exact spectra coincide. Here, motivic Landweber exactness should be understood in the sense of \cite{NSO09}. We can prove the same result for objects of the form $\Gamma_\star(X)$ where $X$ is Landweber exact in the classical sense.

\begin{Cor}\phantomsection\label{effective and very effective covers of motivic analogues of Landweber exact spectra are equivalent}
    Let $X \in {\normalfont \textbf{SH}^{\cl}}$ be bounded below and Landweber exact. Then there is a canonical equivalence
    \[\tf_q(\Gamma_\star(X)) \xrightarrow[]{\simeq} f_q(\Gamma_\star(X)).\]
\end{Cor}

\begin{proof}
    We use the models $\Gamma_\star^q(X)$ and $\tGamma_\star^q(X)$ for the $q$-th effective and very effective covers, respectively. There is a map
    \[\tGamma_\star^q(X) \to \Gamma_\star^q(X)\]
    induced by the transformations $\tau_{\geq w + q} \to \tau_{\geq -\infty} = \text{id}$ and $\tau_{\geq 2q} \to \tau_{\geq -\infty} = \text{id}$ in weights $w \geq q$ and $w \leq q$, respectively.

    By definition of Landweber exactness, we have
    \[X_* \BP \cong \BP_* \BP \otimes_{\BP_*} X_*.\]
    So, in the language of \cite[Appendix A]{Rav86}, $X_* \BP$ is an extended $\BP_*\BP$-comodule. Also in that language, the $E_2$-page of the Adams-Novikov spectral sequence for $X$ is
    \[\text{Cotor}_{\BP_* \BP}^{*, *}(\BP_*, \BP_* X).\]
    Then we can use $\BP_* X \cong X_* \BP$ and \cite[A1.2.8, (b)]{Rav86} to see that this $E_2$-page is entirely concentrated in filtration $0$. So the ANSS for $X$ collapses on $E_2$ and the filtration $0$ line is isomorphic to the homotopy groups of $X$. This implies that the outermost truncations in the definition of $\tGamma_\star^q(X)$ have no effect. The result follows since the map we defined above then induces isomorphisms on all homotopy groups.
\end{proof}

\section{Slices of \texorpdfstring{$\mathbb{C}$}{C}-motivic analogues}\label{section: Slices of C-motivic analogues}

In this section we will use the models for the effective, connective, and very effective covers of a $\mathbb{C}$-motivic analogue $\Gamma_\star(X)$ given in Section \ref{section: Filtered spectrum models for covers of C-motivic analogues} to compute its respective slices. Section \ref{section: effective slices of C-motivic analogues} contains the proof of Theorem \ref{Theorem on slices of motivic analogues} from the introduction.

To make the computations for the effective slices work out nicely as claimed in Theorem \ref{Theorem on slices of motivic analogues} we will have to put some restrictions on the spectrum $X \in \textbf{SH}^{\cl}$, namely that it is bounded below and has even $\MU$ homology. More generally, we can compute the homotopy groups of the effective slices for any bounded below $X$, but that will not lead to a description of the effective slices as simple as in Theorem \ref{Theorem on slices of motivic analogues}.

To give some example use cases of Theorem \ref{Theorem on slices of motivic analogues}, it can be applied to compute the slices of the $\mathbb{C}$-motivic spectra $S$, $\kq$, $\mmf$, $\MGL$, and $\kgl$. These are of the form $\Gamma_\star(X)$ where $X$ is $S$, $\ko$, $\tmf$, $\MU$, and $\ku$, respectively.

\subsection{Effective slices of \texorpdfstring{$\mathbb{C}$}{C}-motivic analogues}\label{section: effective slices of C-motivic analogues}

Let $X \in \textbf{SH}^{\cl}$ and assume $X$ is bounded below so that we can make use of the results in Section \ref{section: Filtered spectrum models for covers of C-motivic analogues}.

\begin{Prop}\phantomsection\label{Filtered spectrum model for effective slice}
    Let $X \in {\normalfont \textbf{SH}^{\cl}}$ be bounded below. Then the $q$-th effective slice of $\Gamma_\star(X)$ is given by the filtered spectrum whose degree $w$ part is
    \[(s_q(\Gamma_\star(X)))_w = \begin{dcases}
    {\normalfont \text{pt}}, & w \geq q + 1\\
    {\normalfont \Tot(\tau_{\geq 2q}^{< 2q + 2}(X \wedge \MU^{\wedge \bullet + 1}))}, & w \leq q
    \end{dcases}\]
    and whose structure maps are trivial for $w \geq q + 1$, and the identity for $w \leq q$.
\end{Prop}

\begin{proof}
    By Proposition \ref{Filtered spectrum model for q-th effective cover} the $q$-th effective cover of $\Gamma_\star(X)$ is modeled by $\Gamma_\star^q(X) \to \Gamma_\star(X)$. Then the $q$-th effective slice $s_q(\Gamma_\star(X))$ is defined via the cofiber sequence
    \[\Gamma_\star^{q + 1}(X) \to \Gamma_\star^q(X) \to s_q(\Gamma_\star(X)).\]
    The first two filtered objects are given by totalizations in each degree. We claim that the third object is too: In our stable setting cofibers are the same as fibers. But fibers are limits, and so too is totalization. So the claim follows from limits commuting with limits.
    
    Therefore $s_q(\Gamma_\star(X))$ is the filtered object given in each degree by the totalization of the cofiber of the underlying map of cosimplicial objects. That yields the claimed description of the slices.
\end{proof}

Using this description we can compute the homotopy groups of $s_q(\Gamma_\star(X))$ via a BKSS. It will be the ANSS for $X$ restricted to the two antidiagonals $2q$ and $2q + 1$. The $d_1$-differentials occur as usual because they start and end on the same antidiagonal. So the BKSS $E_2$-page is the ANSS $E_2$-page restricted to the two antidiagonals. The only other differentials that can occur for degree reasons are $d_2$-differentials whose source lies on the $2q$-th antidiagonal. In general, this complicates matters as the resulting BKSS $E_3$-page is not guaranteed to be the ANSS $E_3$-page. To be more precise, there are three relevant cases to consider:
\begin{itemize}
    \item ANSS $E_2$-page elements on the $2q$-th antidiagonal which support a $d_2$ will also do so in the BKSS,
    \item ANSS $E_2$-page elements on the $2q$-th antidiagonal which get hit by a $d_2$ will survive the BKSS,
    \item ANSS $E_2$-page elements on the $(2q + 1)$-st antidiagonal which support a $d_2$ will be cycles in the BKSS.
\end{itemize}
This leads to a description of the homotopy groups of $s_q(\Gamma_\star(X))$ which is a mix of the $E_2$- and $E_3$-page of the ANSS of $X$. However, under certain hypotheses the ANSS of $X$ avoids many of these cases because it has no $d_2$-differentials. Namely, if one assumes that $X$ has even $\MU$ homology, then by \cite[4.4.2. Proposition]{Rav86} there are no $d_2$-differentials in the ANSS for $X$. More precisely, all odd antidiagonals in the ANSS are trivial. Now the above discussion leads to the following description of the homotopy groups of $s_q(\Gamma_\star(X))$.

\begin{Prop}\phantomsection\label{Homotopy groups of effective slices of motivic analogues}
    Let $X \in {\normalfont \textbf{SH}^{\cl}}$ be bounded below with even $\MU$ homology. Let $E_2^{s, f}$ denote the Adams-Novikov spectral sequence $E_2$-page of $X$ in stem $s$ and filtration $f$. Then
    \[\pi_{*, *}(s_q(\Gamma_\star(X))) \cong \bigoplus_{s + f = 2q} E_2^{s, f}[\tau],\]
    where $E_2^{s, f}$ is placed in stem $s$ and weight $q$.
\end{Prop}

\begin{proof}
    Run the BKSS for the model of the $q$-th slice given in Proposition \ref{Filtered spectrum model for effective slice}
    \[(s_q(\Gamma_\star(X)))_w = \begin{dcases}
    \text{pt}, & w \geq q + 1\\
    \Tot(\tau_{\geq 2q}^{< 2q + 2}(X \wedge \MU^{\wedge \bullet + 1})), & w \leq q.
    \end{dcases}\]
    As $X$ is assumed to have even $\MU$ homology, the BKSS computing the homotopy groups of the slice will degenerate on $E_2$ and all non-trivial elements live on the $2q$-th antidiagonal. Since $(s_q(\Gamma_\star(X)))_w$ is constant in $w$ for $w \leq q$, it follows that the homotopy groups will be free on $\tau$ as claimed.
\end{proof}

Note that $s_0(S) = M\mathbb{Z}$, cf. \cite[Theorem 6.6]{Voe04} or \cite[section 8]{Lev14} or \cite{BE21}, and that any slice $s_q(Y)$ is a module over $s_0(S)$, cf. \cite[Theorem 3.6.22]{Pel11} or \cite[section 6, property (v)]{GRSO12}. Then, to get Theorem \ref{Theorem on slices of motivic analogues} from the above, we would like to invoke a $\mathbb{C}$-motivic analogue of Adams' theorem that a module over the classical Eilenberg-MacLane spectrum $H\mathbb{Z}$ splits as a wedge of generalized Eilenberg-MacLane spectra indexed by its homotopy groups. This motivic analogue of Adams' theorem is Proposition \ref{MZ-modules are wedges of Eilenberg-MacLane spectra}. The following series of lemmas builds up to that.

\begin{Lemma}\phantomsection\label{Realize a group in homotopy}
    Fix $s, w \in \mathbb{Z}$ and let $G$ be an abelian group. Then there exists a spectrum $C \in {\normalfont \textbf{SH}^{\mathbb{C}}}$ such that
    \[\pi_{s, *}(C) \cong G[\tau],\]
    where $G$ is placed in weight $w$.
\end{Lemma}

\begin{proof}
    Choose a free resolution of $G$
    \[0 \to \bigoplus_{j \in J} \mathbb{Z} \xrightarrow[]{r} \bigoplus_{i \in I} \mathbb{Z} \xrightarrow[]{g} G \to 0.\]
    We have
    \[[\bigvee_{j \in J} S^{s,w}, \bigvee_{i \in I} S^{s,w}] \cong \prod_J \bigoplus_I [S^{s, w}, S^{s, w}] \cong \prod_J \bigoplus_I \mathbb{Z} \cong \prod_J \bigoplus_I \text{Hom}_{\mathbb{Z}}(\mathbb{Z}, \mathbb{Z}) \cong \text{Hom}_{\mathbb{Z}}(\bigoplus_J \mathbb{Z}, \bigoplus_I \mathbb{Z}).\]
    Then consider the map of spheres
    \[\bigvee_{j \in J} S^{s, w} \xrightarrow[]{\rho} \bigvee_{i \in I} S^{s,w},\]
    corresponding to $r$. Then by construction we have $\pi_{s, w}(\rho) = r$. Let $C$ be defined by the cofiber sequence
    \[\bigvee_{j \in J} S^{s, w} \xrightarrow[]{\rho} \bigvee_{i \in I} S^{s,w} \xrightarrow[]{\gamma} C.\]
    Because $\pi_{s, *}(S^{s, w}) = \mathbb{Z}[\tau]$ with $\mathbb{Z}$ in weight $w$, the isomorphism $\pi_{s, *}(C) \cong G[\tau]$ with $G$ in weight $w$ follows from the long exact sequence of homotopy groups associated to the cofiber sequence.
\end{proof}

\begin{Lemma}\phantomsection\label{Abstract isomorphism can be realized by a map}
    Fix $s, w \in \mathbb{Z}$. Let $Y \in {\normalfont \textbf{SH}^{\mathbb{C}}}$ and assume $\pi_{s, *}(Y) \cong G[\tau]$ for some abelian group $G$ placed in weight $w$. Let $C$ denote the spectrum from Lemma \ref{Realize a group in homotopy} with $\pi_{s, *}(C) \cong G[\tau]$, again with $G$ in weight $w$. Then there exists a map $\alpha: C \to Y$ such that $\pi_{s, *}(\alpha): G[\tau] \to G[\tau]$ is an isomorphism.
\end{Lemma}

\begin{proof}
    Recall that $C$ was defined as a cofiber
    \[\bigvee_{j \in J} S^{s, w} \xrightarrow[]{\rho} \bigvee_{i \in I} S^{s,w} \xrightarrow[]{\gamma} C.\]
    Apply the functor $[-, Y]$ to this cofiber sequence to obtain the long exact sequence
    \[\dotsc \to [C, Y] \xrightarrow[]{- \circ \gamma} [\bigvee_{i \in I} S^{s,w}, Y] \xrightarrow[]{- \circ \rho} [\bigvee_{j \in J} S^{s,w}, Y] \to \dotsc.\]
    We can rewrite $[\bigvee_{i \in I} S^{s,w}, Y]$ as follows:
    \[[\bigvee_{i \in I} S^{s,w}, Y] \cong \prod_{i \in I}[S^{s, w}, Y] \cong \prod_{i \in I} \pi_{s, w}(Y) \cong \prod_{i \in I} G \cong \prod_{i \in I} \text{Hom}_{\mathbb{Z}}(\mathbb{Z}, G) \cong \text{Hom}_{\mathbb{Z}}(\bigoplus_{i \in I} \mathbb{Z}, G),\]
    and similarly for the set indexed by $J$. Then the long exact sequence is of the form
    \[\dotsc \to [C, Y] \xrightarrow[]{\psi} \text{Hom}_{\mathbb{Z}}(\bigoplus_{i \in I} \mathbb{Z}, G) \xrightarrow[]{\phi} \text{Hom}_{\mathbb{Z}}(\bigoplus_{j \in J} \mathbb{Z}, G) \to \dotsc.\]
    
    We claim that the image of $\psi$ is given by $\text{Hom}_{\mathbb{Z}}(G, G)$. To see this, first apply $\text{Hom}_{\mathbb{Z}}(-, G)$ to the free resolution of $G$
    \[0 \to \text{Hom}_{\mathbb{Z}}(G, G) \xrightarrow[]{g} \text{Hom}_{\mathbb{Z}}(\bigoplus_{i \in I} \mathbb{Z}, G) \xrightarrow[]{f} \text{Hom}_{\mathbb{Z}}(\bigoplus_{j \in J} \mathbb{Z}, G) \to \dotsc.\]
    The map $f$ can be identified with $\phi$ because $f$ is induced by the map $r$ in the free resolution
    \[0 \to \bigoplus_{j \in J} \mathbb{Z} \xrightarrow[]{r} \bigoplus_{i \in I} \mathbb{Z} \xrightarrow[]{g} G \to 0,\]
    and $\phi$ comes from $\rho$ which induces $r$ on homotopy groups. The image of $\psi$ is then determined by exactness
    \[\text{Hom}_{\mathbb{Z}}(G, G) \cong \text{im}(g) = \text{ker}(f) \cong \text{ker}(\phi) = \text{im}(\psi).\]

    Now choose $\alpha \in [C, Y]$ so that it maps to the identity under
    \[[C, Y] \xrightarrow[]{\psi} \text{Hom}_{\mathbb{Z}}(G,G).\]
    We claim that $\alpha$ induces an isomorphism
    \[\pi_{s,w}(\alpha): \pi_{s, w}(C) \xrightarrow[]{\cong} \pi_{s, w}(Y).\]
    To see this, one can check that by construction the map $\pi_{s,w}(\alpha)$ makes the following diagram commute
    \[
    \begin{tikzcd}
        0 \ar[r] & \bigoplus_{j \in J} \mathbb{Z} \ar[r] & \bigoplus_{i \in I} \mathbb{Z} \ar[r] & \pi_{s, w}(Y) \ar[r] & 0\\
        0 \ar[r] & \bigoplus_{j \in J} \mathbb{Z} \ar[r] \ar[u, "="] & \bigoplus_{i \in I} \mathbb{Z} \ar[r] \ar[u, "="] & \pi_{s, w}(C) \ar[r] \ar[u, "\pi_{s, w}(\alpha)"] & 0
    \end{tikzcd}
    \]
    By the five lemma $\pi_{s,w}(\alpha)$ is an isomorphism. Note that by naturality of $\tau$ we get that $\alpha$ more so induces an isomorphism
    \[\pi_{s, *}(\alpha): \pi_{s, *}(C) \xrightarrow[]{\cong} \pi_{s, *}(Y),\]
    as desired.
\end{proof}

\begin{Lemma}\phantomsection\label{Smashing with MZ preserves iso in homotopy}
    Let $C$, $Y$, and $\alpha$ be as in Lemma \ref{Abstract isomorphism can be realized by a map}. Assume $Y$ is an $M\mathbb{Z}$-module. Then
    \[\pi_{s, *}(\normalfont{\text{id}} \wedge \alpha): \pi_{s, *}(M\mathbb{Z} \wedge C) \to \pi_{s, *}(M\mathbb{Z} \wedge Y)\]
    is an isomorphism.
\end{Lemma}

\begin{proof}
    Consider the following commutative diagram
    \[
    \begin{tikzcd}
        C \ar[r, "\alpha"] \ar[d, "\simeq"'] & Y\\
        S \wedge C \ar[d, "h \wedge \text{id}"'] & \\
        M\mathbb{Z} \wedge C \ar[d, "\text{id} \wedge \alpha"'] & \\
        M\mathbb{Z} \wedge Y \ar[ruuu, "\mu"', bend right] & 
    \end{tikzcd}
    \]
    where $h: S \to M\mathbb{Z}$ is the Hurewicz map (the unit of the ring-spectrum structure of $M\mathbb{Z}$), and $\mu: M\mathbb{Z} \wedge Y \to Y$ is the $M\mathbb{Z}$-module multiplication map. All of the maps in this diagram induce isomorphisms on $\pi_{s, *}$:
    \begin{itemize}
        \item For $\alpha$ it is true by construction,
        \item for $h \wedge \text{id}$ it follows because when building a cellular approximation for $M\mathbb{Z}$ we start with $S$ and then continue by attaching a cell to kill the Hopf map $\eta$ (and then inductively attaching cells in larger stems), which keeps $\pi_{0, *}$ unchanged,
        \item for $\mu$ it follows because it is true for $h \wedge \text{id}$,
        \item for $\text{id} \wedge \alpha$ it then follows by commutativity of the diagram.\qedhere
    \end{itemize}
\end{proof}

The following Lemma \ref{Maps between motivic Eilenberg-MacLane spectra are the same as group homomorphisms} is stated without proof in \cite[Lemma A.3]{RSO19}. As we will need to use it, we provide a short proof in the $\mathbb{C}$-motivic setting using filtered spectra.

\begin{Lemma}\phantomsection\label{Maps between motivic Eilenberg-MacLane spectra are the same as group homomorphisms}
    Let $G, G'$ be abelian groups. Then there is an isomorphism
    \[[MG, MG'] \cong \text{\normalfont Hom}_{\mathbb{Z}}(G, G').\]
\end{Lemma}

\begin{proof}
    Use the models $\Gamma_\star(HG)$ and $\Gamma_\star(HG')$ for $MG$ and $MG'$, respectively. We have
    \[\Gamma_w(HG) = \begin{dcases}
        \text{pt}, & w > 0\\
        HG, & w \leq 0
    \end{dcases}\]
    and similarly for $\Gamma_\star(HG')$. So there is a canonical isomorphism
    \[[MG, MG'] \cong [HG, HG'].\]
    The claim then follows from the classical fact
    \[[HG, HG'] \cong \text{Hom}_{\mathbb{Z}}(G, G').\qedhere\]
\end{proof}

\begin{Lemma}\phantomsection\label{MZ smash C is a generalized Eilenberg-MacLane spectrum}
    Let $C$ be as in Lemma \ref{Realize a group in homotopy}. Then there is an equivalence
    \[M\mathbb{Z} \wedge C \xrightarrow[]{\simeq} \Sigma^{s, w} MG.\]
\end{Lemma}

\begin{proof}
    Take the defining cofiber sequence for $C$ from Lemma \ref{Realize a group in homotopy} and smash it with $M\mathbb{Z}$ to obtain the cofiber sequence
    \[\bigvee_{j \in J} \Sigma^{s, w} M\mathbb{Z} \xrightarrow[]{\tilde{\rho}} \bigvee_{i \in I} \Sigma^{s, w} M\mathbb{Z} \xrightarrow[]{\tilde{\gamma}} M\mathbb{Z}\wedge C.\]
    Apply $[-, \Sigma^{s, w} MG]$ to obtain a long exact sequence
    \[
    \dotsc \to [M\mathbb{Z}\wedge C, \Sigma^{s, w} MG] \xrightarrow[]{- \circ \tilde{\gamma}} [\bigvee_{i \in I} \Sigma^{s, w} M\mathbb{Z}, \Sigma^{s, w} MG] \xrightarrow[]{- \circ \tilde{\rho}} [\bigvee_{j \in J} \Sigma^{s, w} M\mathbb{Z}, \Sigma^{s, w} MG] \to \dotsc.
    \]
    The two right-hand sets can be identified with group homomorphisms of the respective underlying groups by Lemma \ref{Maps between motivic Eilenberg-MacLane spectra are the same as group homomorphisms}. In particular, the middle set contains a class $x$ realizing the projection $\bigoplus_{i \in I} \mathbb{Z} \xrightarrow[]{g} G$ from the free resolution of $G$. Then $x$ maps to $0$ on the right because $- \circ \tilde{\rho}$ can be identified with composition by the other map $\bigoplus_{j \in J} \mathbb{Z} \xrightarrow[]{r} \bigoplus_{i \in I} \mathbb{Z}$ from the free resolution of $G$. So $x$ has a preimage $y$ coming from the left. That preimage $y$ fits into the diagram of cofiber sequences
    \[
    \begin{tikzcd}
        \bigvee_{j \in J} \Sigma^{s, w} M\mathbb{Z} \ar[r, "\tilde{\rho}"] \ar[d, "="] & \bigvee_{i \in I} \Sigma^{s, w} M\mathbb{Z} \ar[r, "\tilde{\gamma}"] \ar[d, "="] & M\mathbb{Z} \wedge C \ar[d, "y"] \\
        \bigvee_{j \in J} \Sigma^{s, w} M\mathbb{Z} \ar[r] & \bigvee_{i \in I} \Sigma^{s, w} M\mathbb{Z} \ar[r] & \Sigma^{s ,w} MG
    \end{tikzcd}
    \]
    which implies that it is an equivalence.
\end{proof}

\begin{Prop}\phantomsection\label{MZ-modules are wedges of Eilenberg-MacLane spectra}
    Let $Y \in {\normalfont \textbf{SH}^{\mathbb{C}}}$ be an $M\mathbb{Z}$-module. Fix $w \in \mathbb{Z}$ and for each $s \in \mathbb{Z}$ let $G_s$ be an abelian group. Assume the homotopy groups of $Y$ are of the form
    \[\pi_{s, *}(Y) \cong G_s[\tau]\]
    for every $s \in \mathbb{Z}$, where $G_s$ is placed in weight $w$. Then
    \[Y \simeq \bigvee_{s \in \mathbb{Z}} \Sigma^{s, w} MG_s.\]
\end{Prop}

\begin{proof}
    Fix $s \in \mathbb{Z}$. By Lemmas \ref{Realize a group in homotopy} and \ref{Abstract isomorphism can be realized by a map} there is a spectrum $C_s$ such that
    \[\pi_{s, *}(C_s) \cong G_s[\tau],\]
    and a map $\alpha_s: C_s \to Y$ such that
    \[\pi_{s, *}(\alpha_s): G_s[\tau] \to G_s[\tau]\]
    is an isomorphism.

    Since $Y$ is an $M\mathbb{Z}$-module, Lemma \ref{Smashing with MZ preserves iso in homotopy} implies that
    \[\pi_{s, *}(\text{id} \wedge \alpha_s): \pi_{s, *}(M\mathbb{Z} \wedge C_s) \to \pi_{s,*}(M\mathbb{Z} \wedge Y)\]
    is also an isomorphism.

    By Lemma \ref{MZ smash C is a generalized Eilenberg-MacLane spectrum} we have an equivalence
    \[M\mathbb{Z} \wedge C_s \simeq \Sigma^{s, w} MG_s.\]
    Now we can form the map $\tilde{\alpha}_s$ as the composite
    \[\Sigma^{s, w} MG_s \xrightarrow[]{\simeq} M\mathbb{Z} \wedge C_s \xrightarrow[]{\text{id} \wedge \alpha_s} M\mathbb{Z} \wedge Y \xrightarrow[]{\mu} Y,\]
    where $\mu$ is the $M\mathbb{Z}$-module structure map of $Y$. Recall from the proof of Lemma \ref{Smashing with MZ preserves iso in homotopy} that $\mu$ also induces an isomorphism on $\pi_{s, *}$. Then $\tilde{\alpha}_s$ induces an isomorphism on $\pi_{s, *}$ by construction. Taking the wedge sum over all $s \in \mathbb{Z}$
    \[\bigvee_{s \in \mathbb{Z}} \tilde{\alpha}_s: \bigvee_{s \in \mathbb{Z}} \Sigma^{s, w} MG_s \to Y\]
    yields a map that induces isomorphisms on all homotopy groups because $\pi_{*, *}(MG) \cong G[\tau]$ with $G$ in bidegree $(0, 0)$, as noted in Corollary \ref{Homotopy groups of Eilenberg-MacLane spectra} below. So the map we constructed is the desired equivalence.
\end{proof}

At the end of the proof of Proposition \ref{MZ-modules are wedges of Eilenberg-MacLane spectra} we used the well-known structure of the homotopy groups of $\mathbb{C}$-motivic, $2$-complete Eilenberg-MacLane spectra. We provide a proof of these facts here using a filtered spectrum approach.

\begin{Lemma}\phantomsection\label{homotopy groups of MZ}
    The homotopy groups of $M\mathbb{Z}$ are given by
    \[\pi_{*, *}(M\mathbb{Z}) \cong \mathbb{Z}[\tau],\]
    where $\mathbb{Z}$ is placed in stem $0$ and weight $0$.
\end{Lemma}

\begin{proof}
	Using $M\mathbb{Z} \simeq \Gamma_\star(H\mathbb{Z})$ the claim follows immediately from the BKSS computing the homotopy groups of $\Gamma_\star(H\mathbb{Z})$.
\end{proof}

\begin{Cor}\phantomsection\label{Homotopy groups of Eilenberg-MacLane spectra}
    For any abelian group $G$ the homotopy groups of $MG$ are given by
    \[\pi_{*, *}(MG) \cong G[\tau],\]
    where $G$ is placed in stem $0$ and weight $0$.
\end{Cor}

\begin{proof}
    Choose a free resolution
    \[0 \to \bigoplus_{j \in J} \mathbb{Z} \to \bigoplus_{i \in I} \mathbb{Z} \to G \to 0.\]
    Then consider the associated cofiber sequence of Eilenberg-MacLane spectra
    \[\bigvee_{j \in J} M\mathbb{Z} \to \bigvee_{i \in I} M\mathbb{Z} \to MG.\]
    The induced long exact sequence on homotopy groups implies the claim by Lemma \ref{homotopy groups of MZ}.
\end{proof}

\begin{Cor}[Theorem \ref{Theorem on slices of motivic analogues}]\phantomsection\label{Corollary on slices of motivic analogues}
    Let $X$ be a bounded below classical spectrum with even $\MU$ homology. Let $E_2^{s, f}$ denote the Adams-Novikov spectral sequence $E_2$-page of $X$ in stem $s$ and filtration $f$. Then the $q$-th effective slice of $\Gamma_\star(X)$ is given by
    \[s_q(\Gamma_\star(X)) \simeq \bigvee_{s + f = 2q} \Sigma^{s, q} ME_2^{s,f },\]
    where $M$ is the functor taking a group to the associated $\mathbb{C}$-motivic Eilenberg-MacLane spectrum.
\end{Cor}

\begin{proof}
    Combine Proposition \ref{Homotopy groups of effective slices of motivic analogues} with Proposition \ref{MZ-modules are wedges of Eilenberg-MacLane spectra}.
\end{proof}

Using the above, we obtain many results that have previously been established, and also new results. We list some known computations first that are also implied by Corollary \ref{Corollary on slices of motivic analogues}. Note that the cited results hold in more generality than just in the $\mathbb{C}$-motivic, $2$-complete stable homotopy category, which is where they follow from Corollary \ref{Corollary on slices of motivic analogues}.
\begin{itemize}
    \item The effective slices of $\MGL$ as computed by Hopkins-Morel and written up in \cite{Hoy15},
    \item the effective slices of the sphere spectrum \cite{Lev14},
    \item the effective slices of $\kq$ \cite{ARO20}.
\end{itemize}
More generally than just $\MGL$, we can compute the slices of the motivic analogue of any Landweber exact spectrum satisfying the hypotheses of Corollary \ref{Corollary on slices of motivic analogues}. Slices of motivic Landweber exact spectra have been computed before, see \cite{Spi10} and \cite{Spi12}. Again, those computations hold in more generality than just in the $\mathbb{C}$-motivic, $2$-complete context.

\begin{Cor}
    Let $X \in {\normalfont \textbf{SH}^\mathbb{C}}$ be a Landweber exact spectrum satisfying the hypotheses of Corollary \ref{Corollary on slices of motivic analogues}. Then
    \[s_q(\Gamma_\star(X)) \simeq \Sigma^{2q, q} M(\pi_{2q}(X)).\]
\end{Cor}

\begin{proof}
    For any Landweber exact spectrum $X$ the ANSS $E_2$-page is concentrated in filtration $0$ as we saw in the proof of Corollary \ref{effective and very effective covers of motivic analogues of Landweber exact spectra are equivalent}. So the ANSS collapses on $E_2$ and the filtration $0$ line agrees with the homotopy groups of $X$. Then Corollary \ref{Corollary on slices of motivic analogues} implies the claim.
\end{proof}

We also obtain new results on effective slices. For example, we can compute the slices of the motivic modular forms spectrum $\mmf$. Recall that $\mmf$ is by definition the motivic analogue $\mmf = \Gamma_\star(\tmf\,)$.

\begin{Cor}
    Let $E_2^{s, f}$ denote the Adams-Novikov spectral sequence $E_2$-page of ${\normalfont \tmf}$ in stem $s$ and filtration $f$. Then the $q$-th effective slice of ${\normalfont \mmf}$ is given by
    \[s_q({\normalfont{\mmf\,}}) \simeq \bigvee_{s + f = 2q} \Sigma^{s, q} ME_2^{s, f}.\]
\end{Cor}

\begin{proof}
    Apply Corollary \ref{Corollary on slices of motivic analogues} in the case $\Gamma_\star(\tmf\,) = \mmf$. We check the hypotheses: $\tmf$ is bounded below since it is by definition a connective cover. It has even $\MU$ homology for example by \cite[Corollary 5.2]{Mat16}, or by \cite[Lemma 5.6]{GIKR}.
\end{proof}

A computation of the Adams-Novikov spectral sequence of $\tmf$ can be found in \cite{Bau08}. A chart of the Adams-Novikov $E_2$-page of $\tmf$ can be found in \cite[Figure 3]{IKL+24}. Technically, that chart depicts the $E_2$-page of the $\mathbb{C}$-motivic Adams-Novikov spectral sequence of $\mmf$. However, as discussed in \cite[Section 2.3]{IKL+24}, the $\mathbb{C}$-motivic $E_2$-page is obtained by taking the classical Adams-Novikov $E_2$-page of $\tmf$ and adjoining $\tau$ freely. Since $\tau$-multiples are not plotted in the chart, it can also be interpreted as a chart of the classical $E_2$-page of $\tmf$. The lines in the chart not colored in gray can be ignored, as they describe motivic Adams-Novikov differentials.

\subsection{Connective slices of \texorpdfstring{$\mathbb{C}$}{C}-motivic analogues}\label{section: connective slices of C-motivic analogues}

Let $X \in \textbf{SH}^{\cl}$ and assume $X$ is bounded below so that we can make use of the results in Section \ref{section: Filtered spectrum models for covers of C-motivic analogues}. By Proposition \ref{Filtered spectrum model for q-th connective cover} the $q$-th connective cover of $\Gamma_\star(X)$ is modeled by $\Gamma_\star^{\geq q}(X) \to \Gamma_\star(X)$. Then the $q$-th connective slice $\Gamma_\star(X)_{= q}$ is defined via the cofiber sequence
\[\Gamma_\star^{\geq q + 1}(X) \to \Gamma_\star^{\geq q}(X) \to \Gamma_\star(X)_{= q}.\]

\begin{Prop}\phantomsection\label{Filtered spectrum model for connective slice}
	Let $X \in {\normalfont \textbf{SH}^{\cl}}$ be bounded below. Then the $q$-th connective slice of $\Gamma_\star(X)$ is given by the filtered spectrum whose degree $w$ part is
	\[(\Gamma_\star(X)_{= q})_w = \tau_{= w + q}(\Gamma_w(X))\]
	and whose structure maps are trivial.
\end{Prop}

\begin{proof}
	Recall the definition
	\begin{align*}
		\Gamma_w^{\geq q}(X) &= \tau_{\geq w + q}(\Tot(\tau_{\geq 2w}(X \wedge \MU^{\wedge \bullet + 1})))\\
		&= \tau_{\geq w + q}(\Gamma_w(X))
	\end{align*}
	for the model of the $q$-th connective cover of $\Gamma_\star(X)$. The claimed structure of the $q$-th slice is then immediate from the defining cofiber sequence.
\end{proof}

\subsection{Very effective slices of \texorpdfstring{$\mathbb{C}$}{C}-motivic analogues}

Let $X \in \textbf{SH}^{\cl}$ and assume $X$ is bounded below so that we can make use of the results in Section \ref{section: Filtered spectrum models for covers of C-motivic analogues}. By Proposition \ref{Filtered spectrum model for q-th very effective cover} the $q$-th very effective cover of $\Gamma_\star(X)$ is modeled by $\tGamma_\star^q(X) \to \Gamma_\star(X)$. Then the $q$-th very effective slice $\ts_q(\Gamma_\star(X))$ is defined via the cofiber sequence
\[\tGamma_\star^{q + 1}(X) \to \tGamma_\star^q(X) \to \ts_q(\Gamma_\star(X)).\]

\begin{Prop}\phantomsection\label{Filtered spectrum model for very effective slice}
    Let $X \in {\normalfont \textbf{SH}^{\cl}}$ be bounded below. Then its $q$-th very effective slice is given by the filtered spectrum whose degree $w$ part is
    \[(\ts_q(\Gamma_\star(X)))_w = \begin{dcases}
    {\normalfont \tau_{= w + q}\Tot(\tau_{\geq 2w}(X \wedge \MU^{\wedge \bullet + 1}))}, & w \geq q + 1\\
    {\normalfont \tau_{\geq 2q}^{< 2q + 2} (X)}, & w \leq q
    \end{dcases}\]
    and whose structure maps are trivial for degrees $w \geq q + 2$, the identity for degrees $w \leq q$, and the structure map $(\ts_q(\Gamma_\star(X)))_{q + 1} \to (\ts_q(\Gamma_\star(X)))_{q}$ is the lower horizontal map in the diagram of cofiber sequences
    \[
    \normalfont{
    \begin{tikzcd}
    	\tau_{\geq 2q + 2}(\Tot(\tau_{\geq 2q + 2}(X \wedge \MU^{\wedge \bullet  + 1}))) \ar[r, "="] \ar[d] & \tau_{\geq 2q + 2}(\Tot(\tau_{\geq 2q + 2}(X \wedge \MU^{\wedge \bullet  + 1}))) \ar[d]\\
    	\tau_{\geq 2q + 1}(\Tot(\tau_{\geq 2q + 2}(X \wedge \MU^{\wedge \bullet  + 1}))) \ar[r] \ar[d] & \tau_{\geq 2q}(\Tot(\tau_{\geq 2q}(X \wedge \MU^{\wedge \bullet  + 1}))) \ar[d]\\
    	\tau_{= 2q + 1}(\Tot(\tau_{\geq 2q + 2}(X \wedge \MU^{\wedge \bullet  + 1}))) \ar[r] & \tau_{\geq 2 q}^{< 2q + 2} (X).
    \end{tikzcd}
	}
    \]
\end{Prop}

\begin{proof}
    Recall the filtered spectrum model for the $q$-th very effective cover of a motivic analogue as given in Proposition \ref{Filtered spectrum model for q-th very effective cover}
    \[\tGamma_w^q(X) = \begin{dcases}
    \tau_{\geq w + q} (\Tot(\tau_{\geq 2w} (X \wedge \MU^{\wedge \bullet + 1}))), & w \geq q\\
    \tau_{\geq 2q} (\Tot(\tau_{\geq 2q} (X \wedge \MU^{\wedge \bullet + 1}))), & w \leq q.
    \end{dcases}\]
    The $q$-th very effective slice is given by the cofiber of $\tGamma_w^{q + 1}(X) \to \tGamma_w^q(X)$. The case $w \geq q + 1$ is then immediate.
    
    In the case $w \leq q$ recall from Remark \ref{very effective cover in small weights} that $\tGamma_w^q(X) \cong \tau_{\geq 2q}(X)$ and $\tGamma_w^{q + 1}(X) \cong \tau_{\geq 2q + 2}(X)$. The claimed structure of the slice is again immediate.
    
    The filtered spectrum structure maps are obtained by explicitly writing out the maps between the cofiber sequences defining the slice in the weights under consideration.
\end{proof}

\begin{Rem}
	The filtered spectrum structure map of $\ts_q(\Gamma_\star(X))$ in degrees $q + 1 \to q$ can be described on homotopy groups as follows: In stems $s \neq 2q + 1$ it is trivial. In stem $2q + 1$ it is an injection with image the set of homotopy classes with Adams-Novikov filtration at least $1$. This follows from considering the BKSS associated to $(\ts_q(\Gamma_\star(X)))_{q + 1}$.
\end{Rem}

\begin{Rem}
    A direct consequence of Proposition \ref{Filtered spectrum model for very effective slice} is that the homotopy groups of the $q$-th very effective slice of a motivic analogue are concentrated only in stems $2q$ and $2q + 1$ for weights $w \leq q$, and only in stem $w + q$ for weights $w \geq q + 1$.
\end{Rem}

\begin{Rem}
    Another direct consequence of Proposition \ref{Filtered spectrum model for very effective slice} is that all elements in the homotopy groups of the $q$-th very effective slice in weights $w \leq q + 1$ are $\tau$-torsion-free, and all elements in weights $w \geq q + 2$ are simple $\tau$-torsion.
\end{Rem}

\begin{Rem}\phantomsection\label{Remark on Bachmann's computation of the slices of KQ}
    In \cite{Bac17} the very effective slices of the Hermitian $K$-theory spectrum $\KQ$ are computed. It is a nice exercise to recover the $\mathbb{C}$-motivic case of the computation of the homotopy groups of these slices by using Proposition \ref{Filtered spectrum model for very effective slice} and $\kq = \Gamma_\star(\ko)$, where $\kq$ denotes the very effective cover of $\KQ$. We have to use $\kq$ instead of $\KQ$ so that we can write it as the analogue of a bounded below classical spectrum in order to make use of Proposition \ref{Filtered spectrum model for very effective slice}. But since the very effective slices of $\KQ$ are $4$-periodic, it suffices to consider the slices of $\kq$. Note that \cite{Bac17} uses the term generalized slice for what we call a very effective slice.
\end{Rem}

\section{The effective slice spectral sequence}\label{section: The effective slice spectral sequence}

Let $X \in \textbf{SH}^{\cl}$ be a bounded below classical spectrum with even $\MU$ homology. Using the descriptions of the effective filtration from Proposition \ref{Filtered spectrum model for q-th effective cover} and the effective slices from Proposition \ref{Filtered spectrum model for effective slice} we can compute the effective slice spectral sequence of $\Gamma_\star(X)$ in its entirety. We will see that it is a re-indexing of the Adams-Novikov spectral sequence of the underlying classical spectrum $X$.

Let us first state the definition of the effective slice spectral sequence for $Y \in \textbf{SH}^{\mathbb{C}}$. Consider the tower
\[\dotsc \to f_{q + 1}(Y) \to f_q(Y) \to f_{q - 1}(Y) \to \dotsc.\]
Taking cofibers and applying $\pi_{*, *}$ yields an exact couple
\[
\begin{tikzcd}
    \dotsc \ar[r] & \pi_{*, *}(f_{q + 1}(Y)) \ar[rr] &[-4em] &[-4em] \pi_{*, *}(f_q(Y)) \ar[rr] \ar[dl] &[-4em] &[-4em] \pi_{*, *}(f_{q - 1}(Y)) \ar[r] \ar[dl] & \dotsc\\
     & & \pi_{*, *}(s_q(Y)) \ar[ul] & & \pi_{*, *}(s_{q - 1}(Y)) \ar[ul] & &
\end{tikzcd}
\]
The associated spectral sequence is by definition the effective slice spectral sequence of $Y$. Note that every map in the above diagram preserves the weight, so we can analyze the effective slice spectral sequence one weight at a time.

The goal of this section will be to prove Theorem \ref{Theorem on the effective slice spectral sequence of a nice motivic analogue} below. We state the theorem first, then give a list of the steps involved in its proof, then follow those steps. The proof of Theorem \ref{Theorem on the effective slice spectral sequence of a nice motivic analogue} is given at the end of this section. The reader might notice that we have not given an explicit grading convention for the effective slice spectral sequence. We believe that carrying the additional notation through the statements and proofs does not help the exposition.

\begin{Th}\phantomsection\label{Theorem on the effective slice spectral sequence of a nice motivic analogue}
    Let $X$ be a bounded below classical spectrum with even $\normalfont{\MU}$ homology. Then the $E_1$-page of the effective slice spectral sequence of $\Gamma_\star(X)$ is given by freely adjoining $\tau$ to the $E_2$-page of the Adams-Novikov spectral sequence of $X$.
    
    Furthermore, the differentials in the effective slice spectral sequence and the Adams-Novikov spectral sequence are related as follows: $d_r(x) = \tau^r y$ in the effective slice spectral sequence of $\Gamma_\star(X)$ if and only if $d_{2r + 1}(x) = y$ in the Adams-Novikov spectral sequence of $X$.
\end{Th}

\begin{Rem}
	In the statement of Theorem \ref{Theorem on the effective slice spectral sequence of a nice motivic analogue}, if $x$ describes an element on the $E_2$-page of the ANSS of $X$ in stem $s$ and filtration $f$, then the weight of the class corresponding to $x$ on the $E_1$-page of the effective slice spectral sequence of $\Gamma_\star(X)$ is given by $\frac{s + f}{2}$.
\end{Rem}

\begin{Rem}
	In the case where $X$ is the sphere spectrum and we consider the effective slice spectral sequence in weight $w = 0$, Theorem \ref{Theorem on the effective slice spectral sequence of a nice motivic analogue} recovers a result of Levine \cite[Theorem 1]{Lev15}. We should mention that the cited result of Levine does not assume completion at a prime.
\end{Rem}

The steps involved in the proof of Theorem \ref{Theorem on the effective slice spectral sequence of a nice motivic analogue} are as follows:
\begin{enumerate}
    \item\phantomsection\label{Step 1 for proof of identification of effective slice spectral sequence} Start by identifying the $E_1$-page of the effective slice spectral sequence of $\Gamma_\star(X)$,
    \item\phantomsection\label{Step 2 for proof of identification of effective slice spectral sequence} halve the speed of the spectral sequence by also considering odd values for $n$ in the cosimplicial object $\tau_{\geq n}(X \wedge \MU^{\wedge \bullet + 1})$,
    \item\phantomsection\label{Step 3 for proof of identification of effective slice spectral sequence} observe that the half-speed spectral sequence is the décalage of a truncated ANSS for $X$,
    \item\phantomsection\label{Step 4 for proof of identification of effective slice spectral sequence} use Levine's result \cite[Proposition 6.3]{Lev15} as stated in \cite[Theorem 2.80]{vN25} to compare the décalage spectral sequence to the truncated ANSS,
    \item\phantomsection\label{Step 5 for proof of identification of effective slice spectral sequence} by definition that truncated ANSS will also be the BKSS for $\Gamma_w(X)$.
\end{enumerate}
In step \eqref{Step 1 for proof of identification of effective slice spectral sequence} we start by identifying the $E_1$-page of the effective slice spectral sequence.

\begin{Prop}\phantomsection\label{E_1-page of the effective slice spectral sequence of Gamma_star(X)}
    Let $X$ be a bounded below classical spectrum with even $\normalfont{\MU}$ homology. Then the $E_1$-page of the effective slice spectral sequence of $\Gamma_\star(X)$ is given by freely adjoining $\tau$ to the $E_2$-page of the Adams-Novikov spectral sequence of $X$.
\end{Prop}

\begin{proof}
    The effective filtration of $\Gamma_\star(X)$ can be described using the results of Section \ref{section: filtered spectrum models for effective covers}, meaning we have $f_q(\Gamma_\star(X)) = \Gamma_\star^q(X)$. So the effective slice spectral sequence comes from the exact couple
    \[
    \begin{tikzcd}
        \dotsc \ar[r] & \pi_{*, *}(\Gamma_\star^{q + 1}(X)) \ar[rr] &[-4em] &[-4em] \pi_{*, *}(\Gamma_\star^q(X)) \ar[rr] \ar[dl] &[-4em] &[-4em] \pi_{*, *}(\Gamma_\star^{q - 1}(X)) \ar[r] \ar[dl] & \dotsc\\
         & & \pi_{*, *}(s_q(\Gamma_\star(X))) \ar[ul] & & \pi_{*, *}(s_{q - 1}(\Gamma_\star(X))) \ar[ul] & &
    \end{tikzcd}
    \]
    The claim then follows from Proposition \ref{Homotopy groups of effective slices of motivic analogues}.    
\end{proof}

\begin{Rem}
    Tautologically, the exact couple from the proof of Proposition \ref{E_1-page of the effective slice spectral sequence of Gamma_star(X)} is exact. Using only this fact, one can almost deduce the claim on the effective slice differentials in Theorem \ref{Theorem on the effective slice spectral sequence of a nice motivic analogue}. However, exactness only tells us what happens to a set of elements. What we could deduce is that there are bijections between the sets of elements supporting differentials in the effective slice spectral sequence of $\Gamma_\star(X)$ and in the Adams-Novikov spectral sequence of $X$, and similarly for targets of differentials. But we would not know whether we have the same source and target pairs in the effective slice spectral sequence as in the Adams-Novikov spectral sequence.
\end{Rem}

Recall from the definition of the effective slice spectral sequence that every map in the defining exact couple preserves the weight. So we can analyze the spectral sequence one weight at a time. Let us fix a weight $w$ for the remainder of this section. In weight $w$, the tower giving rise to the exact couple defining the effective slice spectral sequence is given by
\[\dotsc \to \Gamma_w^{q + 1}(X) \to \Gamma_w^q(X) \to \Gamma_w^{q - 1}(X) \to \dotsc.\]
Unpacking the definition, we see that this is the tower
\[\dotsc \to \Gamma_{w + 1}(X) \to \Gamma_{w}(X) \xrightarrow[]{=} \Gamma_w(X) \xrightarrow[]{=} \dotsc,\]
where the transition maps on the left are the ones from $\Gamma_\star(X)$, and on the right we have the identity. For step \eqref{Step 2 for proof of identification of effective slice spectral sequence} we now need to halve the speed of this tower, so that in step \eqref{Step 3 for proof of identification of effective slice spectral sequence} we can identify the associated half-speed spectral sequence as a décalage.

Consider the following adjustment $\zeta_\star$ of $\Gamma_\star$, which in degree $w$ is given by
\[\zeta_w(X) = \Tot(\tau_{\geq w}(X \wedge \MU^{\wedge \bullet + 1})),\]
and the transition maps $\zeta_{w + 1}(X) \to \zeta_w(X)$ are induced by $\tau_{\geq w + 1} \to \tau_{\geq w}$. We want to compare the spectral sequence associated to the tower
\[\dotsc \to \Gamma_{w + 1}(X) \to \Gamma_{w}(X) \xrightarrow[]{=} \Gamma_w(X) \xrightarrow[]{=} \dotsc,\]
with the spectral sequence associated to the tower
\[\dotsc \to \zeta_{2w + 1}(X) \to \zeta_{2w}(X) \xrightarrow[]{=} \zeta_{2w}(X) \xrightarrow[]{=} \dotsc.\]
Let us denote the first spectral sequence by ${}_{\eff_w}E_*^{*, *}(X)$ because it is the effective slice spectral sequence of $X$ in weight $w$, and the second one by ${}_{\zeta_{2w}}E_*^{*, *}(X)$.

\begin{Lemma}\phantomsection\label{iso from Gamma_w to zeta_2w}
    For every $r \geq 1$ there is an isomorphism
    \[{}_{\zeta_{2w}}E_{2r}^{*, *}(X) \cong {}_{\eff_w}E_r^{*, *}(X).\]
    In the spectral sequence ${}_{\zeta_{2w}}E_*^{*, *}(X)$ all differentials of odd index $d_{2r - 1}$ are trivial. The even index differentials $d_{2r}$ on ${}_{\zeta_{2w}}E_{2r}^{*, *}(X)$ correspond to the differentials $d_r$ on ${}_{\eff_w}E_r^{*, *}(X)$ under the above isomorphism.
\end{Lemma}

\begin{proof}
    Note that for all $i \in \mathbb{Z}$ the transition maps $\zeta_{2i}(X) \to \zeta_{2i - 1}(X)$ are equivalences. To see this, first observe that the homotopy groups of both spectra can be computed by a BKSS, which is a truncated version of the ANSS of $X$. The two truncations are below the $2i$-th or below the $(2i - 1)$-st antidiagonal, respectively. But these truncations have the same effect on the ANSS because $X$ is assumed to have even $\MU$ homology, so that odd antidiagonals are trivial.
    
    It follows that in the tower
    \[\dotsc \to \zeta_{2w + 1}(X) \to \zeta_{2w}(X) \xrightarrow[]{=} \zeta_{2w}(X) \xrightarrow[]{=} \dotsc\]
    every second map on the left is an equivalence, and every other map agrees with the appropriate map from the tower
    \[\dotsc \to \Gamma_{w + 1}(X) \to \Gamma_{w}(X) \xrightarrow[]{=} \Gamma_w(X) \xrightarrow[]{=} \dotsc.\]
    The claims then follow by inspecting the associated exact couples and their spectral sequences.
\end{proof}

Now for step \eqref{Step 3 for proof of identification of effective slice spectral sequence} we want to identify the spectral sequence ${}_{\zeta_{2w}}E_*^{*, *}(X)$ as a décalage. Let us first recall the décalage construction.

Let $\Dec$ be the décalage functor defined as follows: If $Z^\bullet$ is a cosimplicial spectrum, then $\Dec_\star(Z^\bullet)$ is the filtered spectrum given in filtration $s$ by
\[\Dec_s(Z^\bullet) = \Tot(\tau_{\geq s} Z^\bullet).\]
The maps in this filtered spectrum are induced by $\tau_{\geq s + 1} \to \tau_{\geq s}$. The décalage spectral sequence of $Z^\bullet$ is by definition the spectral sequence associated to the filtered spectrum $\Dec_s(Z^\bullet)$, i.e. the spectral sequence of the tower
\[\dotsc \to \Dec_{s + 1}(Z^\bullet) \to \Dec_s(Z^\bullet) \to \Dec_{s - 1}(Z^\bullet) \to \dotsc.\]

Our claim for step \eqref{Step 3 for proof of identification of effective slice spectral sequence} is that the décalage spectral sequence of the filtered spectrum
\[Z_{2w}^\bullet = \tau_{\geq 2w}(X \wedge \MU^{\wedge \bullet + 1})\]
is the same as the spectral sequence of the tower
\[\dotsc \to \zeta_{2w + 1}(X) \to \zeta_{2w}(X) \xrightarrow[]{=} \zeta_{2w}(X) \xrightarrow[]{=} \dotsc.\]
The spectral sequence of this tower already has the notation ${}_{\zeta_{2w}}E_*^{*, *}(X)$. Let us denote the décalage spectral sequence of the object $Z_{2w}^\bullet$ by ${}_{\Dec_{2w}}E_*^{*, *}(X)$.

\begin{Lemma}\phantomsection\label{equality of zeta_2w and Dec}
    The spectral sequences ${}_{\zeta_{2w}}E_*^{*, *}(X)$ and ${}_{\normalfont \Dec_{2w}}E_*^{*, *}(X)$ are equal.
\end{Lemma}

\begin{proof}
    The spectral sequence ${}_{\zeta_{2w}}E_*^{*, *}(X)$ comes from the tower
    \[\dotsc \to \zeta_{2w + 1}(X) \to \zeta_{2w}(X) \xrightarrow[]{=} \zeta_{2w}(X) \xrightarrow[]{=} \dotsc.\]
    The spectral sequence ${}_{\Dec_{2w}}E_*^{*, *}(X)$ comes from the tower
    \[\dotsc \to \Dec_{s + 1}(Z_{2w}^\bullet) \to \Dec_s(Z_{2w}^\bullet) \to \Dec_{s - 1}(Z_{2w}^\bullet) \to \dotsc.\]
    By definition this second tower is given by
    \[\dotsc \to \Tot(\tau_{\geq s + 1} Z_{2w}^\bullet) \to \Tot(\tau_{\geq s} Z_{2w}^\bullet) \to \Tot(\tau_{\geq s - 1} Z_{2w}^\bullet) \to \dotsc.\]
    Further expanding the definition of $Z_{2w}^\bullet$ yields
    \[\tau_{\geq s} Z_{2w}^\bullet = \tau_{\geq s} \tau_{\geq 2w}(X \wedge \MU^{\wedge \bullet + 1}) = \begin{dcases}
        \tau_{\geq s}(X \wedge \MU^{\wedge \bullet + 1}), & s \geq 2w\\
        \tau_{\geq 2w}(X \wedge \MU^{\wedge \bullet + 1}), & s \leq 2w.
    \end{dcases}\]
    Recalling the definition of $\zeta_\star(X)$ via
    \[\zeta_w(X) = \Tot(\tau_{\geq w}(X \wedge \MU^{\wedge \bullet + 1}))\]
    shows that the two towers giving rise to the two spectral sequences under consideration are equal. So their spectral sequences are equal too.
\end{proof}

In step \eqref{Step 4 for proof of identification of effective slice spectral sequence} we will now use Levine's result \cite[Proposition 6.3]{Lev15} as stated in \cite[Theorem 2.80]{vN25} to compare the décalage spectral sequence to the Bousfield-Kan spectral sequence of the underlying cosimplicial object. The décalage spectral sequence already has the notation ${}_{\Dec_{2w}}E_*^{*, *}(X)$. Let us denote the Bousfield-Kan spectral sequence of the underlying cosimplicial object $Z_{2w}^\bullet$ by ${}_{Z_{2w}^\bullet}E_*^{*, *}(X)$.

\begin{Lemma}\phantomsection\label{iso from Dec to Z_2w}
    There is an isomorphism of spectral sequences
    \[{}_{\normalfont \Dec_{2w}}E_{r}^{*, *}(X) \cong {}_{Z_{2w}^\bullet}E_{r + 1}^{*, *}(X).\]
\end{Lemma}

\begin{proof}
    This is exactly the statement of \cite[Theorem 2.80]{vN25} applied to our case.
\end{proof}

Step \eqref{Step 5 for proof of identification of effective slice spectral sequence} is just the simple observation that the cosimplicial object $Z_{2w}^\bullet$ is the same one as the one used in the definition of $\Gamma_w(X)$. Again, to fix notation let us denote the Bousfield-Kan spectral sequence of $\Gamma_w(X)$ by ${}_{\Gamma_w}E_{*}^{*, *}(X)$.

\begin{Lemma}\phantomsection\label{iso from Z_2w to Gamma_w}
    The spectral sequences ${}_{Z_{2w}^\bullet}E_*^{*, *}(X)$ and ${}_{\Gamma_w}E_{*}^{*, *}(X)$ are equal.
\end{Lemma}

\begin{proof}
    The spectral sequence ${}_{Z_{2w}^\bullet}E_*^{*, *}(X)$ is the BKSS associated to
    \[\Tot(Z_{2w}^\bullet) = \Tot(\tau_{\geq 2w}(X \wedge \MU^{\wedge \bullet + 1})).\]
    The same is true for ${}_{\Gamma_w}E_{*}^{*, *}(X)$.
\end{proof}

With all of the above steps, we can now prove Theorem \ref{Theorem on the effective slice spectral sequence of a nice motivic analogue}.

\begin{proof}[Proof of Theorem \ref{Theorem on the effective slice spectral sequence of a nice motivic analogue}]
    The claim on the $E_1$-page is Proposition \ref{E_1-page of the effective slice spectral sequence of Gamma_star(X)}. The claim on the differentials follows from the chain of isomorphisms
    \[{}_{\eff_w}E_r^{*, *}(X) \cong {}_{\zeta_{2w}}E_{2r}^{*, *}(X) = {}_{\Dec_{2w}}E_{2r}^{*, *}(X) \cong {}_{Z_{2w}^\bullet}E_{2r + 1}^{*, *}(X) = {}_{\Gamma_w}E_{2r + 1}^{*, *}(X),\]
    described in Lemmas \ref{iso from Gamma_w to zeta_2w}, \ref{equality of zeta_2w and Dec}, \ref{iso from Dec to Z_2w}, and \ref{iso from Z_2w to Gamma_w}. The fact that the effective differentials are $d_r(x) = \tau^r y$, instead of just $d_r(x) = y$, is forced by degree reasons: If $x$ is on the $2q$-th antidiagonal of the ANSS for $X$, then $y$ will be on the $2(q + r)$-th antidiagonal. So $x$ is detected in the homotopy of $s_q(\Gamma_\star(X))$ and $y$ is detected in the homotopy of $s_{q + r}(\Gamma_\star(X))$. Because of Corollary \ref{Corollary on slices of motivic analogues} that means in the slice $E_1$-page $x$ has weight $q$ and $y$ has weight $q + r$. Therefore $x$ must hit $\tau^r y$ with its slice $d_r$-differential. In other words, $\tau^r y$ is the version of $y$ that lives in the same weight as $x$, and since the effective slice differentials preserve weights, this is the differential that must occur.
\end{proof}

\bibliographystyle{alpha}
\bibliography{bib}

\newcommand{\etalchar}[1]{$^{#1}$}
\begin{thebibliography}{GRS{\O}12}

\bibitem[AR{\O}20]{ARO20}
Alexey Ananyevskiy, Oliver R\"ondigs, and Paul~Arne {\O}stv{\ae}r.
\newblock On very effective hermitian {$K$}-theory.
\newblock {\em Math. Z.}, 294(3-4):1021--1034, 2020.

\bibitem[Bac17]{Bac17}
Tom Bachmann.
\newblock The generalized slices of {H}ermitian {$K$}-theory.
\newblock {\em J. Topol.}, 10(4):1124--1144, 2017.

\bibitem[Bau08]{Bau08}
Tilman Bauer.
\newblock Computation of the homotopy of the spectrum {\tt tmf}.
\newblock In {\em Groups, homotopy and configuration spaces}, volume~13 of {\em
  Geom. Topol. Monogr.}, pages 11--40. Geom. Topol. Publ., Coventry, 2008.

\bibitem[BE21]{BE21}
Tom Bachmann and Elden Elmanto.
\newblock Voevodsky's slice conjectures via {H}ilbert schemes.
\newblock {\em Algebr. Geom.}, 8(5):626--636, 2021.

\bibitem[BHS22]{BHS22}
Robert Burklund, Jeremy Hahn, and Andrew Senger.
\newblock Galois reconstruction of {A}rtin-{T}ate $\mathbb{R}$-motivic spectra.
\newblock preprint,
  \href{https://doi.org/10.48550/arXiv.2010.10325}{\nolinkurl{
  arXiv:2010.10325}}, 2022.

\bibitem[CQ21]{CQ21}
Dominic~Leon Culver and J.~D. Quigley.
\newblock {$kq$}-resolutions {I}.
\newblock {\em Trans. Amer. Math. Soc.}, 374(7):4655--4710, 2021.

\bibitem[DI05]{DI05}
Daniel Dugger and Daniel~C. Isaksen.
\newblock Motivic cell structures.
\newblock {\em Algebr. Geom. Topol.}, 5:615--652, 2005.

\bibitem[GIKR22]{GIKR}
Bogdan Gheorghe, Daniel~C. Isaksen, Achim Krause, and Nicolas Ricka.
\newblock {$\mathbb{C}$}-motivic modular forms.
\newblock {\em J. Eur. Math. Soc. (JEMS)}, 24(10):3597--3628, 2022.

\bibitem[GRS{\O}12]{GRSO12}
Javier~J. Guti\'errez, Oliver R\"ondigs, Markus Spitzweck, and Paul~Arne
  {\O}stv{\ae}r.
\newblock Motivic slices and coloured operads.
\newblock {\em J. Topol.}, 5(3):727--755, 2012.

\bibitem[Hea19]{Hea19}
Drew Heard.
\newblock On equivariant and motivic slices.
\newblock {\em Algebr. Geom. Topol.}, 19(7):3641--3681, 2019.

\bibitem[Hoy15]{Hoy15}
Marc Hoyois.
\newblock From algebraic cobordism to motivic cohomology.
\newblock {\em J. Reine Angew. Math.}, 702:173--226, 2015.

\bibitem[IKL{\etalchar{+}}24]{IKL+24}
Daniel~C. Isaksen, Hana~Jia Kong, Guchuan Li, Yangyang Ruan, and Heyi Zhu.
\newblock The {$\mathbb{C}$}-motivic {A}dams-{N}ovikov spectral sequence for
  topological modular forms.
\newblock {\em Adv. Math.}, 458:Paper No. 109966, 47, 2024.

\bibitem[Isa19]{Isa19}
Daniel~C. Isaksen.
\newblock Stable stems.
\newblock {\em Mem. Amer. Math. Soc.}, 262(1269):viii+159, 2019.

\bibitem[IWX23]{IWX}
Daniel~C. Isaksen, Guozhen Wang, and Zhouli Xu.
\newblock Stable homotopy groups of spheres: from dimension 0 to 90.
\newblock {\em Publ. Math. Inst. Hautes \'Etudes Sci.}, 137:107--243, 2023.

\bibitem[Lev14]{Lev14}
Marc Levine.
\newblock A comparison of motivic and classical stable homotopy theories.
\newblock {\em J. Topol.}, 7(2):327--362, 2014.

\bibitem[Lev15]{Lev15}
Marc Levine.
\newblock The {A}dams-{N}ovikov spectral sequence and {V}oevodsky's slice
  tower.
\newblock {\em Geom. Topol.}, 19(5):2691--2740, 2015.

\bibitem[Lur09]{HTT}
Jacob Lurie.
\newblock {\em Higher topos theory}, volume 170 of {\em Annals of Mathematics
  Studies}.
\newblock Princeton University Press, Princeton, NJ, 2009.

\bibitem[LWX25]{LWX25}
Weinan Lin, Guozhen Wang, and Zhouli Xu.
\newblock On the {L}ast {K}ervaire {I}nvariant {P}roblem.
\newblock preprint,
  \href{https://doi.org/10.48550/arXiv.2412.10879}{\nolinkurl{arXiv:2412.10879}},
  2025.

\bibitem[Mat16]{Mat16}
Akhil Mathew.
\newblock The homology of tmf.
\newblock {\em Homology Homotopy Appl.}, 18(2):1--29, 2016.

\bibitem[Mor99]{Mor99}
Fabien Morel.
\newblock Th\'eorie homotopique des sch\'emas.
\newblock {\em Ast\'erisque}, (256):vi+119, 1999.

\bibitem[Mor05]{Mor05}
Fabien Morel.
\newblock The stable {${\mathbb A}^1$}-connectivity theorems.
\newblock {\em $K$-Theory}, 35(1-2):1--68, 2005.

\bibitem[MV99]{MV99}
Fabien Morel and Vladimir Voevodsky.
\newblock {${\bf A}^1$}-homotopy theory of schemes.
\newblock {\em Inst. Hautes \'Etudes Sci. Publ. Math.}, (90):45--143, 1999.

\bibitem[NRZ26]{NRZ26}
Ahina Nandy, Oliver Röndigs, and Egor Zolotarev.
\newblock Slices of the special linear algebraic cobordism spectrum.
\newblock preprint,
  \href{https://doi.org/10.48550/arXiv.2606.05020}{\nolinkurl{arXiv:2606.05020}},
  2026.

\bibitem[NS{\O}09]{NSO09}
Niko Naumann, Markus Spitzweck, and Paul~Arne {\O}stv{\ae}r.
\newblock Motivic {L}andweber exactness.
\newblock {\em Doc. Math.}, 14:551--593, 2009.

\bibitem[Pel11]{Pel11}
Pablo Pelaez.
\newblock Multiplicative properties of the slice filtration.
\newblock {\em Ast\'erisque}, (335):xvi+289, 2011.

\bibitem[Pst23]{Pst23}
Piotr Pstr\k{a}gowski.
\newblock Synthetic spectra and the cellular motivic category.
\newblock {\em Invent. Math.}, 232(2):553--681, 2023.

\bibitem[Rav86]{Rav86}
Douglas~C. Ravenel.
\newblock {\em Complex cobordism and stable homotopy groups of spheres}, volume
  121 of {\em Pure and Applied Mathematics}.
\newblock Academic Press, Inc., Orlando, FL, 1986.

\bibitem[RS{\O}19]{RSO19}
Oliver R\"ondigs, Markus Spitzweck, and Paul~Arne {\O}stv{\ae}r.
\newblock The first stable homotopy groups of motivic spheres.
\newblock {\em Ann. of Math. (2)}, 189(1):1--74, 2019.

\bibitem[S{\O}12]{SO12}
Markus Spitzweck and Paul~Arne {\O}stv{\ae}r.
\newblock Motivic twisted {$K$}-theory.
\newblock {\em Algebr. Geom. Topol.}, 12(1):565--599, 2012.

\bibitem[Spi10]{Spi10}
Markus Spitzweck.
\newblock Relations between slices and quotients of the algebraic cobordism
  spectrum.
\newblock {\em Homology Homotopy Appl.}, 12(2):335--351, 2010.

\bibitem[Spi12]{Spi12}
Markus Spitzweck.
\newblock Slices of motivic {L}andweber spectra.
\newblock {\em J. K-Theory}, 9(1):103--117, 2012.

\bibitem[vN25]{vN25}
Sven van Nigtevecht.
\newblock An introduction to filtered and synthetic spectra.
\newblock preprint,
  \href{https://doi.org/10.48550/arXiv.2509.21127}{\nolinkurl{
  arXiv:2509.21127}}, 2025.

\bibitem[Voe02a]{Voe02a}
Vladimir Voevodsky.
\newblock Open problems in the motivic stable homotopy theory. {I}.
\newblock In {\em Motives, polylogarithms and {H}odge theory, {P}art {I}
  ({I}rvine, {CA}, 1998)}, volume 3, I of {\em Int. Press Lect. Ser.}, pages
  3--34. Int. Press, Somerville, MA, 2002.

\bibitem[Voe02b]{Voe02b}
Vladimir Voevodsky.
\newblock A possible new approach to the motivic spectral sequence for
  algebraic {$K$}-theory.
\newblock In {\em Recent progress in homotopy theory ({B}altimore, {MD},
  2000)}, volume 293 of {\em Contemp. Math.}, pages 371--379. Amer. Math. Soc.,
  Providence, RI, 2002.

\bibitem[Voe03]{Voe03}
Vladimir Voevodsky.
\newblock Motivic cohomology with {${\bf Z}/2$}-coefficients.
\newblock {\em Publ. Math. Inst. Hautes \'Etudes Sci.}, (98):59--104, 2003.

\bibitem[Voe04]{Voe04}
Vladimir Voevodsky.
\newblock On the zero slice of the sphere spectrum.
\newblock {\em Tr. Mat. Inst. Steklova}, 246:106--115, 2004.

\bibitem[Voe11]{Voe11}
Vladimir Voevodsky.
\newblock On motivic cohomology with {$\mathbf Z/l$}-coefficients.
\newblock {\em Ann. of Math. (2)}, 174(1):401--438, 2011.

\end{thebibliography}

\end{document}